\documentclass[11pt]{article}

\usepackage[T1]{fontenc}
\usepackage[english]{babel}
\usepackage{amsmath}
\usepackage{amsthm}
\usepackage{mathtools}
\usepackage{amsfonts}
\usepackage{enumerate}
\usepackage{graphicx}
\usepackage{float}
\usepackage{subcaption} 
\usepackage{tikz}
\usepackage{mathtools}

\usetikzlibrary{arrows,decorations.markings}

\mathtoolsset{showonlyrefs}

\numberwithin{equation}{section}

\newtheorem{theorem}{Theorem}[section]
\newtheorem{lemma}[theorem]{Lemma}

\newtheorem{defn}[theorem]{Definition}

\newtheorem{RHP}{Riemann-Hilbert Problem}[section]

\theoremstyle{remark}

\newtheorem*{rem}{Remark}

\newcommand{\eps}{\varepsilon}
\newcommand{\NN}{\mathbb{N}}

\newcommand{\RR}{\mathbb{R}}
\newcommand{\CC}{\mathbb{C}}
\newcommand{\UC}{\mathbb{T}}
\newcommand{\Ordo}{\mathcal{O}}
\newcommand{\F}{\mathcal{F}}
\newcommand{\E}{\mathbb{E}}
\renewcommand{\L}{\mathbb{L}}
\DeclareMathOperator{\V}{Var}
\DeclareMathOperator{\Var}{Var}
\DeclareMathOperator{\im}{Im}

\newcommand{\N}{\text{N}}
\renewcommand{\d}{\,\mathrm{d}}
\newcommand{\e}{\mathrm{e}}
\renewcommand{\i}{\mathrm{i}}
\DeclareMathOperator{\Tr}{Tr}

\title{Mesoscopic fluctuations for the thinned Circular Unitary Ensemble}
\author{Tomas Berggren\footnote{Department of Mathematics, Royal Institute of Technology (KTH), Stockholm Lindstedsv\"agen 25, SE-10044, Sweden. Email: tobergg@kth.se.  Supported  by the  grant KAW 2010.0063 from the Knut and Alice Wallenberg Foundation.} 
\and Maurice Duits\footnote{Department of Mathematics, Royal Institute of Technology (KTH), Stockholm Lindstedsv\"agen 25, SE-10044, Sweden. Email: duits@kth.se.  Supported  by the  Swedish Research Council (VR) Grant no.\ 2012-3128.} 
}

\begin{document}

\maketitle

\begin{abstract}
In this paper we study the asymptotic behavior of mesoscopic fluctuations for the thinned Circular Unitary Ensemble. The effect of thinning is that the eigenvalues  start to decorrelate. The decorrelation is  stronger on the larger scales than on the smaller scales.  We investigate this behavior by studying mesoscopic linear statistics. There are two regimes depending on the scale parameter and the thinning parameter.  In one regime we obtain  a CLT of a classical type and in the other regime we retrieve the  CLT for CUE. The two regimes are separated by a critical line. On the critical line the limiting fluctuations are no longer Gaussian, but described by  infinitely divisible laws. We argue that this transition phenomenon is universal by  showing that the same transition and their laws appear for  fluctuations of the thinned sine process in a growing box.  The proofs are based on a Riemann-Hilbert problem for integrable operators.

\end{abstract}

\section{Introduction and statement of results}

   Let $\{{\rm e}^{ {\rm i} \theta_j}\}_{j=1}^n$ be the eigenvalues of a unitary matrix taken  randomly  from the CUE of size $n \in \mathbb N$ (the group of $n \times n$ unitary matrices equipped with Haar measure). Then the \emph{thinned} CUE eigenvalue process is obtained by removing each $\theta_j$ independently with a probability $1-\gamma \in (0,1)$. The purpose of this paper is to investigate the effect of  thinning on the limiting behavior of the  \emph{mesoscopic fluctuations} as $n \to \infty$.  
   
We  analyze the fluctuations by means of linear statistics.    That is, for $\alpha \in (0,1)$ and a sufficiently smooth function $f$ with compact support we consider the random variable
\begin{equation}
\label{eq:deflinestat} 
X_n^{(\alpha,\gamma)}(f)= \sum_j f(\theta_j n^{1-\alpha}).
\end{equation}
 Here the $\{\theta_j\}$ come from the point process obtained by a  thinning, with parameter $\gamma$, of the  eigenangles of a randomly chosen $n \times n$  CUE matrix. This also explains the upper index $\gamma$ in the notation.  Note that the number of eigenvalues that contribute to the linear statistic is random. Since $f$ has compact support only eigenangles on interval of lengths of order $\sim n^{\alpha-1}$ contribute to $X_n^{(\alpha,\gamma)}(f)$. These are the mesoscopic scales, whereas $\alpha=0$ and $\alpha=1$ correspond to the micro- and macroscopic scales.   The purpose of the paper is  to study  the behavior of 
$$X_n^{(\alpha,\gamma)}(f)- \E X_n^{(\alpha,\gamma)}(f),$$
as $n \to \infty$ and how the thinning effects this behavior. 

The thinning of a point process is a classical construction and we refer to, e.g., \cite{DV} for background.  Roughly speaking, its effect is that  the points become more uncorrelated and it can thus be used for  creating a family of point processes that define a cross-over from the original process and a Poisson process, see e.g. \cite[Prop. 11.3.I]{DV}. The eigenvalues of a CUE matrix are strongly correlated and as a result the limiting behavior of the fluctuations as $n \to \infty$ have a very different character when compared to uncorrelated systems. By thinning the CUE we obtain a natural transition between those types of systems. 

In the context of random matrix theory, thinning was first discussed by  Bohigas and Pato \cite{BP1,BP2} for the  sine process. Recently, Bothner, Deift,  Its and Krasovsky \cite{BDI}  computed very detailed asymptotics for the gap probabilities for the thinned sine process.  For the thinned CUE similar results were obtained  by Charlier and Claeys in \cite{CC}. Here we will have have a different perspective and analyze mesoscopic linear statistics. The results in this paper are an addition to the recent developments on various models from random matrix type where similar transitions were studied,  such as free fermions in non-zero temperature by Johansson and Lambert \cite{JL},   random  band matrices by Erd\H{o}s and Knowles  \cite{KE1,KE2} and Dyson's Brownian motion by one of us together with Johansson \cite{DJ}.  Although the laws describing these transitions appear to be depending on the model at hand, one aspect seems to be universal: If we start from the uncorrelated situation, then in the transition  the  fluctuations on the smaller scales, i.e. $\alpha$ closer to $0$,  stabilize earlier to random matrix type of statistics  than those on the larger scales, i.e. $\alpha$ closer to $1$. 

Also in the model of the thinned CUE eigenvalues process we will see that  the effect of thinning is felt stronger on the larger scales then the smaller scales. In fact, if we thin only a little  and take $\gamma= \gamma_n  \to  1$ as $n \to \infty$, then the effect of the thinning on the smallest scales is washed out and we observe the usual random matrix statistics in the limit. However, if we still take $\gamma_n$ such that $n(1-\gamma_n) \to \infty$, then the leading asymptotic behavior of the global fluctuations turns out to be the same as in the case that the $\theta_j$ are independent.  Somewhere in the intermediate scales a transition from one behavior to the other takes place and it is that transition that we want to describe in this paper.

Before we state our main results we first recall some results for the behavior of linear statistics for the CUE without thinning, i.e. $\gamma=1$. In that case it is known that the variance of the linear statistic remains bounded
$\Var X_n^{(\alpha,1)}(f)= O(1)$ as $n \to \infty$ and has the following limit 
$$\lim_{n\to\infty} \Var X_n^{(\alpha,1)}(f) = \int_\RR |\omega| |\mathcal F f(\omega)|^2 {\rm d}\omega,$$
where $\mathcal F$ stands for the Fourier transform
$$\mathcal F f(\omega)= \frac{1}{2 \pi} \int_\RR {\rm e}^{ - {\rm i} x\omega} f(x) {\rm d} x.$$
 The fact that the variance remains bounded is remarkable since the number of points that contribute to the linear statistic is of order $\sim n^\alpha$.  Indeed, if the $\theta_j$'s would be independent then the variance would therefore be of order $\sim n^\alpha$. Due to the bounded variance, it is somewhat surprising that the fluctuations of the linear statistics still obey a Central Limit Theorem (=CLT). That is, 
\begin{equation} \label{eq:mesoscopicCLTCUE}
X_n^{(\alpha,1)}(f)- \E \left[X_n^{(\alpha,1)}(f)\right]
\overset{\mathcal D}{\to } N\left(0,\int_\RR |\omega| |\mathcal F f(\omega)|^2 {\rm d}x\right).
\end{equation}
The first proof of this result is due to Soshnikov \cite{Sosh} and is the analogue of an analogous CLT on the macrosopic scale. On the macrosopic scale it is a direct consequence of the Strong Szeg\H{o} Limit Theorem, see for example  \cite{Diaconis03} for a survey.  The mesoscopic CLT \eqref{eq:mesoscopicCLTCUE} is a universal limit that is expected to hold in the bulk for many random matrix models. In the last years, there has been much interest to prove this universality in a variety of settings \cite{BD,FKS,KH,Lambert}.

Let us now return to the thinned CUE. Then a  computation, cf. \eqref{eq:section_preliminaries:variance_of_thinned_CUE},
shows that the variance of the linear statistic can be written as 
\begin{equation} \label{eq:variancegamma}
 \V[X_n^{(\alpha,\gamma)}(f)] = \frac{ n^{\alpha} \gamma(1-\gamma)}{2\pi}\int_{-n^{1-\alpha} \pi}^{n^{1-\alpha}\pi}f( x)^2dx + \gamma^2\V[X_n^{(\alpha,1)}(f)].
 \end{equation}
Here we clearly see two terms competing with each other. The first term on the right-hand side is close to the variance that we would expect if the $\theta_j$ was taken from the Poisson point process with mean density $ \frac{ n^{\alpha} \gamma(1-\gamma)}{2\pi} $, equal for big $ n $ if $ f $ has compact support, and the second term is the variance for the CUE without thinning. The second term is $\sim 1$ and hence we see that when $n^\alpha(1-\gamma_n) \to \infty$ we have a growing variance and expect a CLT of classical type. When  $n^\alpha(1-\gamma_n) \to 0$, we expect  a CLT of CUE type.  The following theorem confirms that expectation. 

\begin{theorem}\label{th:transCUEfirsttwocases}
Let $0 < \alpha,\delta <1$, $\kappa>0$ and set $\gamma=\gamma_n=1-\kappa n^{-\delta}$. Then for any $f: \RR \to \RR$ with compact support and such that
\begin{equation}\label{eq:optimalcondition}
\int |\mathcal F f(\omega)|^2(1+ |\omega|) {\rm d} \omega < \infty,
\end{equation}
we have 
\begin{enumerate}
\item  If $\alpha < \delta$, then, as $n \to \infty$,
$$X_n^{(\alpha,\gamma_n)}(f)- \E \left[X_n^{(\alpha,\gamma_n)}(f)\right]
\overset{\mathcal D}{\to } \mathcal N\left(0,\int_\RR |\omega| |\mathcal F f(\omega)|^2 {\rm d}x\right).$$
\item If $\alpha > \delta$, then, as $n \to \infty$, 
$$n^{-(\alpha-\delta)/2}\left(X_n^{(\alpha,\gamma_n)}(f)- \E\left[X_n^{(\alpha,\gamma_n)}(f)\right]\right)
\overset{\mathcal D}{\to } \mathcal N\left(0, \frac{\kappa}{2\pi} \int_\RR f(x)^2 {\rm d}x\right)$$
\end{enumerate}
\end{theorem}
The proof of this theorem will be given in Section \ref{Proof of CUE}.

\begin{figure}[t]
\begin{center}
\begin{tikzpicture}[scale=0.9]
\draw[->,thick] (-.3,0)--(6.5,0);
\draw[->,thick] (0,-.3)--(0,6.5);
\draw[-,very thick] (0,0) --(6,6);
\draw (-0.6,-0.4) node {$(0,0)$};
\draw (-.3,6) node {$1$};
\draw (6,-.3) node {$1$};
\draw (0,6.7) node {$\alpha$};
\draw (6.7,6) node {\small{$\alpha=\delta $}};
\draw (7.3,5.1) node{$\mathcal L_{f,\kappa,\sigma_f}$} ;
\draw (7.2,0) node {$\delta$};
\draw[-] (6,-.1)--(6,.1);
\draw[-] (-.1,6)--(.1,6);
\draw[help lines,dashed] (0,6)--(6,6)--(6,0);
\draw (-2,0) node{Microscopic};
\draw (-2,6) node{Macroscopic};
\draw (0,-1) node{ $\sim 1$ thinning};
\draw (6,-1) node{ $\sim 1/n$ thinning};
\draw (1.8,4.0) node{classical CLT} ;
\draw (4,1.5) node{CUE CLT} ;

\end{tikzpicture}
\caption{$\alpha\delta$-Diagram presenting the different regimes. The transition takes place on the critical line $\alpha=\delta$ and the limiting fluctuations are no longer Gaussian.}
\label{fig:phase}
\end{center}
\end{figure}
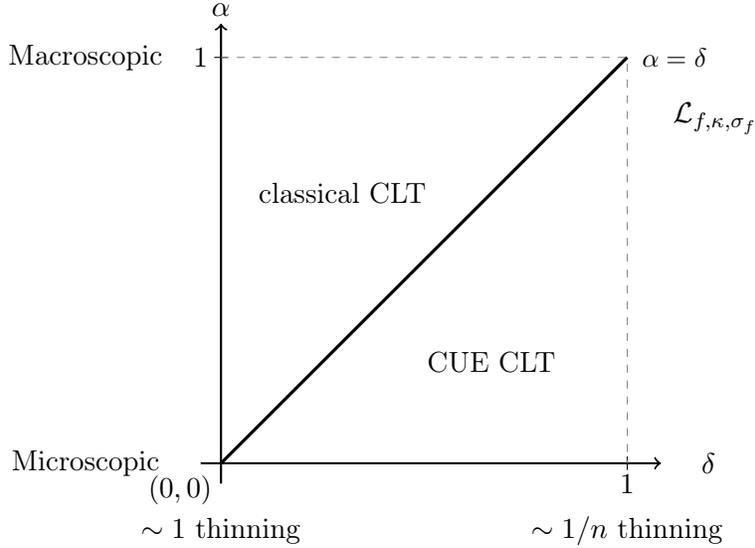

The two terms on the right-hand side of \eqref{eq:variancegamma}  are in balance when 
$$n^\alpha(1-\gamma_n) \to \kappa$$
for some $\kappa>0$ and this is when the transition from one type of behavior to the other takes place. As it turns out the limiting fluctuations are no longer Gaussian, but are described by so-called \emph{infinitely divisible laws.} These are random variables that, for each $n \in \NN$, can be written as the sum of $n$ independent and identically distributed random variables.  The L\'evy-Khintchine Theorem  \cite[Th. 3.8.2]{durrett} tells us that a law $X$ is infinitely divisible if and only if  the characteristic function has the form 
$$\E [\exp  {\rm i} \lambda X]= \exp\left( {\rm i } \beta \lambda - \sigma^2 \lambda^2/2- \int_{\RR\setminus \{0\}} \left( {\rm e} ^{  {\rm i }  \lambda x}-1- \frac{{\rm i} \lambda  x}{1+x^2} \right) {\rm d} \mu(x)\right),$$
where $\beta \in \RR, \sigma^2\geq 0$ and $\mu$ is a measure such that $\mu(\{0\})=0$ and 
$$\int \frac{x^2}{1+x^2} {\rm d} \mu(x) <\infty.$$
The measure $\mu$ is usually  referred to as the \emph{L\'evy measure.}  
The infinitely divisible laws that we encounter are indexed by the function $f$ in the linear statistic and have the form
$$\E [\exp \lambda \mathcal L_{f,\kappa,\sigma}]= \exp\left(\sigma^2\lambda^2/2+ \kappa \int_\RR\left( {\rm e} ^{-\lambda f(x)}-1+\lambda  f(x) \right) {\rm d} x\right).$$
Hence the L\'evy measure $\mu$ is given by the push-forward by $-f$ of the measure $\kappa {\rm d}x$ restricted to the support of $f$ (we recall that the support is compact). That we have $f(x)$ instead of $f(x)/(1+f(x)^2)$ in the integrand is just a shift in the value of $\beta$ in the L\'evy-Khintchine representation. The functions $f$ that we will allow have a slightly stronger smoothness condition than \eqref{eq:optimalcondition}. In particular they will be continuous. 
\begin{theorem} \label{th:transCUEthirdcase}
Let $0 < \alpha<1$, $\kappa>0$ and set $\gamma=\gamma_n=1-\kappa n^{-\alpha}$. For any $f: \RR \to \RR$ with compact support and such that
$$\int_{\mathbb R} |\mathcal Ff(\omega)|^2(1+ |\omega|)^{1+\eps} {\rm d} \omega < \infty,$$
for some $\eps>0$, 
we have 
$$X_n^{(\alpha,\gamma_n)}(f)- \E\left[X_n^{(\alpha,\gamma_n)}(f)\right]
\overset{\mathcal D}{\to } \mathcal L_{f,\kappa/2 \pi,\sigma_f},$$
as $n \to \infty$, where $\sigma_f^2= \int_{\mathbb R} |\omega| |\mathcal F f(\omega)|^2 {\rm d} \omega$. 
\end{theorem}
\begin{rem}
Note that for $\kappa=0$ the law $\mathcal L_{f,0,\sigma_f}$ is the Gaussian corresponding to the CUE limit. For $\kappa \to \infty$ we have that $ (2\pi)^{1/2} \kappa^{-1/2} \mathcal L_{f,\kappa/2 \pi,\sigma_f}$ converges in distribution to a Gaussian corresponding to the classical limit.  So the laws we get for finite $\kappa$ interpolate between the Gaussians from  the two different regimes. 
\end{rem}
The proof of this theorem will given in Section \ref{Proof of CUE}.

The  proof that we present here is based on the fact that the eigenangles of a CUE matrix form a determinantal point process. We recall that a simple point process on the interval $ E \subset \RR $, finite or infinite, is determinantal if there exists a kernel $ K:E \times E \rightarrow \RR $ such that for any test function $ \phi $ we have 
\begin{equation}\label{eq:section_preliminaries:def_of_determinantal_point_process}
 \E\left[\prod_k(1+\phi(x_k))\right] = \sum_{m=0}^{\infty}\frac{1}{m!}\int_{E^m}\det( K(x_k,x_{\ell}))_{k,\ell=1}^{m}\prod_{j=1}^m\phi(x_j)dx_j.
\end{equation}
 As it turns out  the thinning of any determinantal point process is  again a determinantal point process. In fact, the kernel changes in a trivial way, 
 \begin{equation}\label{eq:thinnedkernel}
 K_{\gamma}= \gamma K. 
 \end{equation}
The CUE eigenangle process is a determinantal point process on $[-\pi,\pi]$ with kernel
\begin{equation}\label{eq:CUEkernel}
\hat K_n(\theta_1, \theta_2)= \frac{\sin n (\theta_1-\theta_2)/2}{2 \pi\sin(\theta_1-\theta_2)/2}.
\end{equation}
Since we are dealing with mesoscopic statistics we preferred to state our results using the eigenangles. However, in the proof we will use the eigenvalues, which form a determinantal point process on the unit circle with kernel 
$$  K_n(z,w) = \frac{z^{n}w^{-n}-1}{2\pi i(z-w)}.$$
The benefit of this expression is that the kernel is an example of an integrable operator and thus we can use  the Riemann-Hilbert approach for finding asymptotics of the resolvent of  $\gamma  K_n$ as $n \to \infty$. This will allow us to compute the asymptotic behavior of the moment generating function of the linear statistic, which is well-known to be a Fredholm-determinant. The proofs presented here follow that of the proof given in \cite{Deift99}. A key difference is that we are working with mesoscopic scales which complicates the  opening of the lens. We overcome this problem by making the opening of the lens $n$-dependent and let the lips converge to the circle as $n \to \infty$ in a precise manner. This also requires careful estimates in the error terms.  We note that this shrinking of the lips of the lens appeared already in the paper \cite[Sec. 6.2]{DIK}. There it was used to extend the  proofs of their results for analytic functions to  smooth functions by an approximation argument. In our proof it is not only the regularity  of the function that requires the shrinking, but more the fact that we are  dealing with mesoscopic scales.

Although we only work with the CUE, we strongly believe that the approach in this paper would work for a more general class of determinantal point process for which the kernel defines an integrable operator. For instance, the eigenvalues of  a Unitary Ensemble on the real line form a determinantal point process with kernel given by the Christoffel-Darboux kernel for the corresponding orthogonal polynomials.  After inserting the detailed asymptotics for those  polynomial derived in \cite{DKMVZ1,DKMVZ2} and a careful computation, we believe that the same transition can be proved for this class of models for any point in the \emph{bulk}. We leave this as an open problem and only provide evidence for this universality by proving  that the same laws appear when we consider  linear satistics for the thinned in a \emph{growing box} for the sine process. We believe that sine kernel asymptotics are a key force behind Theorems \ref{th:transCUEfirsttwocases} and \ref{th:transCUEthirdcase}. 

 We introduce a parameter $L$ and define 
$$X_L^{(\gamma)}(f)= \sum f(x_j /L),$$
where the sum is over all $x_j$ taken randomly from the thinned sine process, i.e. the determinantal point process with kernel $$\gamma K_{sine}(x,y)= \gamma \frac{\sin(x-y)}{\pi(x-y)}. $$
We think of $L$ as a large parameter here. Since $f$ has compact support the linear statistics only takes points into account that lie in intervals with size of order $\sim L$.  We then scale the thinning parameter with
$$\gamma= \gamma_L=1-\frac{\kappa}{L^{\delta}}.$$
Then we have the following equivalent of Theorem \ref{th:transCUEfirsttwocases}. 
\begin{theorem} \label{th:transsinefirsttwocases}Let $\kappa> 0$ and take $\gamma=1-\frac{\kappa}{L^\delta}$. 
Then for any $f: \RR \to \RR$ with compact support and such that
\begin{equation}\label{eq:optimalconditionsine}
\int_{\mathbb R} |\mathcal F f(\omega)|^2(1+ |\omega|) {\rm d} \omega < \infty,
\end{equation}
we have the following cases:
\begin{itemize}
\item If $\delta >1$, then 
$$X_L^{(\gamma_L)} (f) -\E X_L^{(\gamma_L)} (f) \overset{\mathcal D}{\to } \mathcal N\left(0, \int_{\mathbb R} |\omega| |\mathcal F f(\omega)|^2  {\rm d} \omega\right),$$
as $L\to \infty$. 
\item If $0 <\delta <1$, then 
$$L^{(1-\delta)/2} \left(X_L^{(\gamma_L)} (f) -\E X_L^{(\gamma_L)} (f)\right) \overset{\mathcal D}{\to } \mathcal N\left(0, \frac{\kappa}{\pi} \int_{\mathbb R} f(x)^2   {\rm d} x \right),$$
as $L\to \infty$. 
\end{itemize}
\end{theorem}
The following is the equivalent of Theorem \ref{th:transCUEthirdcase}.
\begin{theorem} \label{th:transsinethirdcase} Let $\delta=1$, so that $\gamma_L=1-\kappa/L$ for $\kappa>0$. Then, for any $f: \RR \to \RR$ with compact support and such that
$$\int_{\mathbb R} |\mathcal F f(\omega)|^2(1+ |\omega|)^{1+\eps} {\rm d} \omega < \infty,$$
for some $\eps>0$, 
we have
$$X_L^{(\gamma_L)}(f)- \E\left[X_L^{(\gamma_L)}(f)\right]
\overset{\mathcal D}{\to }  \mathcal L_{f,\kappa/\pi,\sigma_f},$$
as $L \to \infty$, where $\sigma_f^2= \int_{\mathbb R} |\omega| |\mathcal F f(\omega)|^2 {\rm d} \omega$. 
\end{theorem}

The proofs of Theorem \ref{th:transsinefirsttwocases} and \ref{th:transsinethirdcase} will be given in Section \ref{proof of Sine}. 

\subsubsection*{Overview of the rest of the paper}

The rest of the paper is organized as follows. In Section 2 we discuss preliminaries on determinantal point process and recall the Riemann-Hilbert problem for  integrable operators, from which we will derive the asymptotic behavior of the moment generating function of the linear statistics. Then in Section 3 we prove Theorems \ref{th:transCUEfirsttwocases} and \ref{th:transCUEthirdcase}  concerning the thinned CUE. In Section 4 we deal with the thinned sine process and prove Theorems \ref{th:transsinefirsttwocases} and \ref{th:transsinethirdcase}. Since the proofs for the thinned sine process are similar to the thinned CUE, and at certain points even simpler, we will allow ourselves to  be brief in our discussion.

\subsubsection*{Acknowledgements}

We thank Kurt Johansson,  Gaultier Lambert and Boualem Djehiche for fruitful and inspiring discussions.  We also thank the anonymous referees for many valuable comments that helped improving the presentation of the paper. 
\section{Preliminaries}
In this section we will recall the main tools that will be used.  In particular we will setup the Riemann-Hilbert problem that we will analyze asymptotically in the next sections. 

\subsection{Determinantal point process and linear statistics}
We start by recalling some basic notions from determinantal point processes.  For more details on determinantal point processes we refer to \cite{AGZ,Borodin,Johansson05,Sosh}.

 If the kernel $ \phi K $ that characterize the point process in \eqref{eq:section_preliminaries:def_of_determinantal_point_process}, defines an operator from $ \mathbb L^2(E) $ to $ \mathbb  L^2(E) $ given by 
\begin{equation}
 \phi K(h)(x) = \phi(x)\int_Eh(y)K(x,y)\, {\rm d} y,
\end{equation}
then the right hand side of \eqref{eq:section_preliminaries:def_of_determinantal_point_process} defines a Fredholm determinant of the operator. We can therefore write \eqref{eq:section_preliminaries:def_of_determinantal_point_process} as  
\begin{equation}\label{eq:section_preliminaries:def_of_fredholm_determinant}
 \E\left[\prod_{k}(1+\phi(x_k))\right] = \det(I+\phi K)_{\mathbb L^2(E)}.
\end{equation}
For a determinantal point process the characteristic function of the linear statistic for a test function  $ h:\RR \rightarrow \RR $, i.e.
\begin{equation}
X(h) = \sum h(x_k),
\end{equation}
takes a particularly nice form. Indeed, by taking $ \phi = {\rm e}^{ {\rm i} t h} -1$ in \eqref{eq:section_preliminaries:def_of_fredholm_determinant}, we find
\begin{equation}\label{eq:momentgtofredholm}
\E \left[ {\rm e}^{ {\rm i} t  X(h)}\right] 
= \det \left(I+ ({\rm e}^{ {\rm i} t h}-1) K\right)_{\mathbb L_2(E)}.
\end{equation}
From here  we can find useful expressions for the expectation and variance of the linear statistic, giving 
\begin{equation}
 \E[X(h)] = \int_{E}h(x)K(x,x)\d x
\end{equation}
and
\begin{equation}
 \V[X(h)] = \int_{E}h(x)^2K(x,x)\d x - \int_{E}\int_{E}h(x)h(y)K(x,y)K(y,x) \, {\rm d} x {\rm d}y,
\end{equation}
see also, e.g. \cite[Sec. 4.2]{AGZ}. Note that if we have $\int_E K(x,y) K(y,x)  \d y = K(x,x)$ we can write alternatively 
\begin{equation}
 \V[X(h)] =\frac12 \int_{E}\int_{E}(h(x)-h(y))^2 K(x,y)K(y,x) \, {\rm d} x {\rm d}y.
\end{equation}
It is well known that the CUE has a determinantal structure with kernel \eqref{eq:CUEkernel}. The expectation and variance of the linear statistic of the (non-thinned) CUE is given by
\begin{equation}\label{eq:section_preliminaries:expectation_of_CUE}
 \E_\text{CUE}[X(h)] = n \hat{h}(0)
\end{equation}
and
\begin{align}\label{eq:section_preliminaries:variance_of_CUE}
 \V_\text{CUE}[X(h)] = 
  & \frac{1}{8\pi^2}\int_{-\pi}^\pi\int_{-\pi}^\pi (h(x)-h(y))^2\frac{\sin^2(\frac{n}{2}(x-y))}{\sin^2(\frac{1}{2}(x-y))}\d x\d y\\
&=  \sum_{k=-\infty}^\infty \min(n,|k|)|\hat{h}(k)|^2 
\end{align}
where $ n $ is the dimension of the CUE and $\hat h(k)$ is the $k$-th Fourier coefficient. The last equality can be proved by standard Fourier arguments, See the beginning of the proof of \cite[Theorem 8.4.5]{Pastur11}.

The sine process has also a determinantal structure with kernel $K(x,y)=\sin(x-y)/\pi(x-y)$. The expectation and variance of the linear statistic are given by
\begin{equation}
 \E_\text{sine}[X(h)] = \frac{1}{\pi}\int_\RR h(\xi)\d\xi
\end{equation}
and 
\begin{equation}\label{eq:section_preliminaries:variance_of_Sine}
 \V_\text{sine}[X(h)] = \frac{1}{2\pi^2}\int_\RR \int_\RR \left(\frac{h(x)-h(y)}{x-y}\right)^2\sin^2(x-y)\d x\d y.
\end{equation}

The thinning preserves the determinantal structure of the process and the kernel is changing by a factor $ \gamma $, cf. \eqref{eq:thinnedkernel}. The expectation and variance of the linear statistic of the thinned CUE are then given by
\begin{equation}\label{eq:section_preliminaries:mean_of_thinned_CUE}
 \E_\text{Thinned CUE}[X(h)] = \gamma \E_\text{CUE}[X(h)]
\end{equation}
and 
\begin{align}\label{eq:section_preliminaries:variance_of_thinned_CUE}
 \V_\text{Thinned CUE}[X(h)] = & \frac{n\gamma}{2\pi}\int_{[-\pi,\pi]}h(x)^2 \d x \\
  & - \gamma^2\int_{[-\pi,\pi]}\int_{[-\pi,\pi]}h(x)h(y)K_{n,1}(x,y)K_{n,1}(y,x) \, {\rm d} x {\rm d}y\\
  = & \frac{ n \gamma(1-\gamma)}{2\pi}\int_{-\pi}^{\pi}h(x)^2\d x + \gamma^2\V_\text{CUE}[X(h)]
\end{align}
respectively. The same type of relation between the thinned and non-thinned sine process holds. That is,
\begin{equation}\label{eq:section_preliminaries:mean_of_thinned_Sine}
 \E_\text{Thinned sine}[X(h)] = \gamma \E_\text{sine}[X(h)]
\end{equation}
and 
\begin{equation}\label{eq:section_preliminaries:variance_of_thinned_Sine}
 \V_\text{Thinned sine}[X(h)] = \frac{\gamma(1-\gamma)}{\pi}\int_\RR h(x)^2\d x + \gamma^2\V_\text{sine}[X(h)].
\end{equation}
We will consider the thinned process on a mesoscopic scale for the CUE and in a growing box for the sine process. The effect on the expectation and variance of the  linear statistic is a dilation of the function $h$. Note that \eqref{eq:variancegamma} follows by taking $h(x)=f(n^{1-\alpha} x)$ and a change of variables.

In the rest of the paper we will drop the subscript on the measure, instead it will be clear what process we consider. 

\subsection{Integrable operators and Riemann-Hilbert problems} 

Let $ K $ be the operator defined by the kernel given from the thinned CUE or the thinned sine process, that is $  K = K_{n,\gamma} $ or $ K = K_{\text{sine},\gamma} $. We write
\begin{equation}\label{eq:section_preliminaries:integrable_kernel}
 K(z,w) = \left\{
 \begin{matrix}
  \frac{f^{(1)}(z)^Tf^{(2)}(w)}{z-w}, & z \neq w, \\
  f^{(1)'}(z)^Tf^{(2)}(w), & z = w,
 \end{matrix}\right.
\end{equation}
with
\begin{align}
 f^{(1)}(z) & = \frac{\gamma}{2\pi i}(z^n,1)^T, \\
 f^{(2)}(w) & = (w^{-n},-1)^T,
\end{align}
for $ z,w \in \UC $ for the CUE and 
\begin{align}
 f^{(1)}(z) & = \frac{\gamma}{2\pi i}(e^{iz},e^{-iz})^T, \\
 f^{(2)}(w) & = (e^{-iw},-e^{iw})^T,
\end{align}
for $ z,w \in \RR $ for the sine process. The rest of the preliminaries are stated for the kernels defined above. 

One approach to analyze \eqref{eq:section_preliminaries:def_of_fredholm_determinant} is to take the logarithmic derivative,
\begin{multline}\label{eq:section_preliminaries:logarithmic_derivative}
 \frac{\partial}{\partial t}\log \det(I-t\phi K)= -\frac{1}{t}\Tr\left((I-t\phi K)^{-1}t\phi K\right).\\
 = -\frac{1}{t}\int_{E}((I-t\phi K)^{-1}t\phi K)(x,x)dx,
\end{multline}
as long as both sides make sense, see \cite[eq. (67)]{Deift00}.

With this approach we need to get hold of the resolvent in an explicit way, which in general can be a difficult task. However, the form of \eqref{eq:section_preliminaries:integrable_kernel} shows that the  operators defined by the CUE kernel and the sine kernel  are examples of  integrable operators. For integrable operators it holds (under mild conditions) that the resolvent of the operator is again an integrable operator and the kernel can be expressed in terms of a solution to a Riemann-Hilbert problem. We refer to \cite{Deift99,Its90,DIZ97} (and in particular \cite[Ex. 3]{Deift99}) for an excellent discussion on this topic. Consider the following Riemann-Hilbert problem. For simplicity, we will assume some analyticity on the function $ \phi $.

\begin{RHP}\label{RHP:section_preliminaries:initial_problem}
Let $ \Gamma_m $ be the unit circle or the real line and let $ \phi $ be an analytic function in a neighborhood of each point $ x \in \Gamma_m $. Find an $ m:\CC\backslash \Gamma_m \rightarrow \CC^{2\times 2} $ such that
\begin{itemize}
 \item $ m $ is analytic in $ \CC \backslash \Gamma_m $
 \item $ m_+(x) = m_-(x)J_m(x) $  for $ x \in \Gamma_m $ where
 \begin{equation}
  J_m(x) = I - 2\pi i t\phi(x)f^{(1)}(x)f^{(2)}(x)^T
 \end{equation}
 \item $ m(z) = I + \Ordo(z^{-1}) $ as $ z \rightarrow \infty $
\end{itemize}
where $ m_+ $ and $ m_- $ are the limits of $ m $ as $ z $ tends to $ \Gamma_m $ from the left and right side respectively and $ f^{(1)} $ and $ f^{(2)} $ are defined by the kernel \eqref{eq:section_preliminaries:integrable_kernel}.
\end{RHP}
For more details and background on Riemann-Hilbert problems we refer to \cite{Deift99,Deift00} and the references therein. 

If there is a unique solution $ m $ to the Riemann-Hilbert problem above, then, with
\begin{equation}\label{eq:section_preliminaries:capital_f_one}
 F^{(1)}(x) = m_+(x)f^{(1)}(x)
\end{equation}
and
\begin{equation}\label{eq:section_preliminaries:capital_f_two}
 F^{(2)}(x) = (m_+(x)^{-1})^Tf^{(2)}(x),
\end{equation}
the kernel of the resolvent of $ t\phi K $, at least if $ \|t\phi\|_{L^2(E)} < 1 $, is given by
\begin{equation}
 (I-t\phi K)^{-1}t\phi K(x,y) = \left\{
 \begin{matrix}
  t\phi(x) \frac{F^{(1)}(x)^TF^{(2)}(y)}{x-y}, & x \neq y \\
  t\phi(x)F^{(1)'}(x)^TF^{(2)}(x), & x = y
 \end{matrix}\right.
\end{equation}
where $ (I-t\phi K)^{-1}t\phi K(x,y) $ is the kernel of the resolvent $ (I-t\phi K)^{-1}t\phi K $. The extra condition on $ \phi $ is to make sure that the resolvent exists as a bounded operator and that it has a kernel. We want to use this in \eqref{eq:section_preliminaries:logarithmic_derivative} and integrate with respect to $ t $ from zero to one. If $ E = \RR $ it requires some decay on $ \phi $ to make sure that the right hand side of \eqref{eq:section_preliminaries:logarithmic_derivative} makes sense, quadratic decay is sufficient. We summarize with the following lemma. 
\begin{lemma}[cf. \cite{Deift99} eq. (62) and (69)]\label{lem:section_preliminaries:characteristic_function_as_integral}
Let $ K = K_{n,\gamma} $ or $ K = K_{\text{sine},\gamma} $ defined by \eqref{eq:section_preliminaries:integrable_kernel} and $ E \subset \CC $ be the set where $ K $ is defined, that is, $ E = \UC $ if $ K = K_{n,\gamma} $ and $ E = \RR $ if $ K = K_{\text{sine},\gamma} $. Let $ \phi $ be an analytic function in some neighborhood of $ E $ such that $ \|\phi\|_{\L^{\infty}(E)} < 1 $. If $ E = \RR $ then we require $ \phi $ to have sufficient decay on $ \RR $. Then
\begin{equation}\label{eq:section_preliminaries:log_of_determinant_integral_representation}
 \log \det(I-\phi K) = -\int_0^1\int_{E}\phi(x)F^{(1)'}(x)^TF^{(2)}(x)dxdt.
\end{equation}
where $ F^{(1)} $ and $ F^{(2)} $ are defined by \eqref{eq:section_preliminaries:capital_f_one} and \eqref{eq:section_preliminaries:capital_f_two}.
\end{lemma}

\section{Proof of Theorem \ref{th:transCUEfirsttwocases} and \ref{th:transCUEthirdcase}}\label{Proof of CUE}

To prove Theorem \ref{th:transCUEfirsttwocases} and \ref{th:transCUEthirdcase}, we will prove that there exists a disk around the origin, such that for compact sets in the disc the moment generating function converges uniformly to the moment generating function of the limiting distribution. That is, we want to prove that
\begin{equation}
 \E\left[\e^{\lambda n^{-s}(X_n^{(\alpha,\gamma_n)}(f)-\E[X_n^{(\alpha,\gamma_n)}(f)])}\right] \rightarrow M(\lambda)
\end{equation}
as $ n \rightarrow \infty $ where  $ s = \max(0,(\alpha-\delta)/2) $, and $ M $ is the moment generating function of the limiting distribution (a Gaussian in case $\alpha\neq \delta$). Given a $ \lambda $ close to the origin, we will consider (recall \eqref{eq:momentgtofredholm})
\begin{equation}
 \E\left[\e^{\lambda n^{-s}X_n^{(\alpha,\gamma_n)}(f)}\right] = \det(I+(\e^{\lambda n^{-s}f_{\alpha}}-1)K_{n,\gamma_n})_{\L^2([-\pi,\pi))} 
\end{equation}
where $ f_{\alpha}(\theta) = f\left(\theta n^{1-\alpha}\right) $.  We will use Lemma \ref{lem:section_preliminaries:characteristic_function_as_integral}  to analyze the  Fredholm determinant on the right-hand side asymptotically. To this end, we perform a Deift/Zhou steepest descent analysis on the associated Riemann-Hilbert problem. Many similar steepest descent analyses exist in the literature. We do not attempt to give a complete list of references  but only mention   \cite{Deift99} as an excellent introduction to the method which is also close to our setting (e.g. \cite[Example 3]{Deift99} contains a proof of the Strong Szeg\H{o} Limit Theorem). An important difference with \cite{Deift99} is that we are working with mesoscopic scales and, as we will see, this complicates the opening of the lens.

An important first issue is that $ f_{\alpha} $ is not analytic in a neighborhood of $ [-\pi,\pi) $. This  is inconvenient in the Deift/Zhou steepest descent analysis. We will therefore replace $ f $ by  an analytic approximation, by truncating  its Fourier series viewed as a function on the circle.  This we will discuss first.

\subsection{Approximation to an analytic function}\label{subsection:section_proof_of_CUE:approximation_to_an_analytic_function}

Let $ f $ be a real valued function with support in $ [-\pi,\pi] $ such that
\begin{equation}\label{eq:section_proof_of_CUE:epsilon_bound}
 \int_\RR|\F(f)(\xi)|^2(1+|\xi|)^{1+\epsilon} \d\xi < \infty
\end{equation}
for some $ \epsilon > 0 $. Take $f_\alpha(\theta)= f (\theta  n^{1-\alpha})$ and  define
\begin{equation}\label{eq:section_proof_of_CUE:def_of_approximated_function}
f_{\alpha,N}(z) = \sum_{k = -N}^{N}\hat{f}_{\alpha}(k)z^k
\end{equation}
where 
\begin{equation}
 \hat{f}(k) = \frac{1}{2\pi}\int_{-\pi}^{\pi}f(\theta)\e^{-\i k\theta}\d\theta
\end{equation} 
is the $ k $th Fourier coefficient.

Note that, since $ f $ has support in $ [-\pi,\pi] $, we can express \eqref{eq:section_proof_of_CUE:def_of_approximated_function} with the Fourier transform instead of the Fourier coefficient. That is, with a change of variable, 
\begin{equation}\label{eq:section_proof_of_CUE:fourier_coefficient_in_terms_of_fourier_transform}
 \hat{f}_{\alpha}(k) = \frac{1}{n^{1-\alpha}}\F(f)\left(\frac{k}{n^{1-\alpha}}\right).
\end{equation}
This will be used to obtain different kind of bounds on $ f_{\alpha,N} $.

A crucial point in our analysis is that $N$ will depend on $n$ when taking the limit $n \to \infty$.  We will specify a more precise choice of $N$ later, but for now we take $N$ such that 
\begin{equation}\label{eq:section_proof_of_CUE:def_N}
 n^{1-\alpha} \ll N \ll n^{\min(1,\frac{3}{2}(1-\alpha))},
\end{equation}
which is sufficient in the beginning of our discussion. 

We will need the following well-known lemma. 
\begin{lemma} \label{lem:momenttovariance}
Let $X$ and $Y$ be real random variables such that $\E X= \E Y$. Then 
$$\left|\mathbb E [{\rm e}^{{\rm i} t  X} ]-\mathbb E [{\rm e}^{{\rm i} t  Y} ] \right| \leq  |t| \sqrt{ \V [X-Y]},$$
for $t \in \mathbb R$. 
\end{lemma}
\begin{proof}
This follows easily from 
$$
\left|\mathbb E [{\rm e}^{{\rm i} t  X} ]-\mathbb E [{\rm e}^{{\rm i} t  Y} ] \right|\leq\mathbb E \left[  \left|{\rm e}^{{\rm i} t  (X-Y)} -1\right|   \right] 
\leq |t| \mathbb E \left[  |X-Y|  \right]  \leq  |t|  \sqrt{\V [X-Y]},
$$
for $t \in \mathbb R$. \end{proof}

The following lemma tells us that  \eqref{eq:section_proof_of_CUE:def_of_approximated_function} is a good approximation.
\begin{lemma}\label{lem:section_proof_of_CUE:correct_approximation_of_f}
 Let $ f:\RR \rightarrow \RR $ have support in $ [-\pi,\pi] $ such that \eqref{eq:section_proof_of_CUE:epsilon_bound} holds. 
  Define $ f_{\alpha,N} $ by \eqref{eq:section_proof_of_CUE:def_of_approximated_function} and let $ s = \max(0,(\alpha-\delta)/2) $. If
 \begin{equation}
  n^{-s}\left(\sum_k f_{\alpha,N}\left(\e^{i\theta_k}\right) - \E\left[\sum_k f_{\alpha,N}\left(\e^{i\theta_k}\right)\right]\right) \rightarrow X
 \end{equation}
 as $ n \rightarrow \infty $ in distribution, for some random variable $ X $, where $ \{\theta_k\}_k$ are taken from the thinned $ CUE $ process, then
 \begin{equation}
  n^{-s}\left(\sum_k f_{\alpha}(\theta_k) - \E\left[\sum_k f_{\alpha}(\theta_k)\right]\right) \rightarrow X
 \end{equation}
 as $ n \rightarrow \infty $ in distribution.
\end{lemma}

\begin{proof}
We will denote the function $ \theta\mapsto f_{\alpha,N}\left(\e^{\i n^{\alpha-1} \theta}\right) $ as $ f_N $. With this notation the lemma states that if
\begin{equation}
  n^{-s}\left(X_n^{(\alpha,\gamma_n)}(f_N) - \E\left[X_n^{(\alpha,\gamma_n)}(f_N)\right]\right) \rightarrow X
 \end{equation}
 as $ n \rightarrow \infty $, then
 \begin{equation}
  n^{-s}\left(X_n^{(\alpha,\gamma_n)}(f) - \E\left[X_n^{(\alpha,\gamma_n)}(f)\right]\right) \rightarrow X
 \end{equation}
 as $ n \rightarrow \infty $. By Lemma \ref{lem:momenttovariance}  it  remains to  show that the variance $\V\left[n^{-s}X_n^{(\alpha,\gamma_n)}(f - f_N)\right]$ tends to zero. 

By \eqref{eq:section_preliminaries:variance_of_thinned_CUE}, \eqref{eq:section_proof_of_CUE:fourier_coefficient_in_terms_of_fourier_transform} and \eqref{eq:section_preliminaries:variance_of_CUE}
 \begin{align}
 & \V[n^{-s}X_n^{(\alpha,\gamma_n)}(f-f_N))] \\
 = & n^{1-2s}\frac{\gamma_n(1-\gamma_n)}{2\pi}\int_{[-\pi,\pi]}(f_{\alpha}(\theta)-f_{\alpha,N}(\e^{i\theta}))^2\d\theta + \gamma_n^2\V[n^{-s}X_n^{(\alpha,1)}(f-f_N)] \\
 = & n^{\alpha-2s}\frac{\gamma_n(1-\gamma_n)}{2\pi}\frac{1}{n^{1-\alpha}}\sum_{|k|>N}\left|\F(f)\left(\frac{k}{n^{1-\alpha}}\right)\right|^2 \\ &  \qquad \qquad  + \gamma_n^2n^{-2s}\frac{1}{n^{1-\alpha}}\sum_{|k|>N}\frac{\min(n,|k|)}{n^{1-\alpha}}\left|\F(f)\left(\frac{k}{n^{1-\alpha}}\right)\right|^2.
 \end{align}
 Since $ n^{\alpha-2s}\gamma_n(1-\gamma_n) $ is bounded we need to estimate the tail to see that the first term tends to zero. Note that
\begin{align}\label{eq:section_proof_of_CUE:L2_bound}
 \frac{1}{n^{1-\alpha}}\sum_{k=-\infty}^\infty\left|\F(f)\left(\frac{k}{n^{1-\alpha}}\right)\right|^2 & = \int_\RR|\F(f)(\xi)|^2\d\xi \\
 & \leq \int_\RR|\F(f)(\xi)|^2(1+|\xi|)^{1+ \epsilon}\d\xi < \infty
\end{align}
since both sides are equal to the $ \L^2 $-norm of $ f $. For any $ M>0 $
\begin{align}\label{eq:section_proof_of_CUE:vanishing_L2_tail}
 & \limsup_n \frac{1}{n^{1-\alpha}}\sum_{|k|>N}\left|\F(f)\left(\frac{k}{n^{1-\alpha}}\right)\right|^2 \\
 = & \limsup_n\left(\frac{1}{n^{1-\alpha}}\sum_{k=-\infty}^\infty\left|\F(f)\left(\frac{k}{n^{1-\alpha}}\right)\right|^2 -  \frac{1}{n^{1-\alpha}}\sum_{k=-\N}^N\left|\F(f)\left(\frac{k}{n^{1-\alpha}}\right)\right|^2\right) \\ 
  \leq & \int_\RR|\F(f)(\xi)|^2\d\xi - \liminf_n\frac{1}{n^{1-\alpha}}\sum_{k=-n^{1-\alpha}M}^{n^{1-\alpha}M}\left|\F(f)\left(\frac{k}{n^{1-\alpha}}\right)\right|^2 \\
  = & \int_{|\xi|>M}|\F(f)(\xi)|^2\d\xi
\end{align}
 where the last equality is true since the second term is a Riemann sum and the Fourier transform of $ f $ is continuous. Hence the tail tends to zero as $ n \to \infty $. Since
 \begin{equation}\label{eq:section_proof_of_CUE:sobolev_norm_in_two_ways}
  \frac{1}{4\pi^2}\int_\RR\int_\RR\left(\frac{f_\alpha(x)-f_\alpha(y)}{x-y}\right)^2\d x\d y = \int_\RR|\xi||\F(f_\alpha)(\xi)|^2\d\xi,
 \end{equation}
(see \cite[Theorem 7.12]{Lieb01}) which is scale invariant, and since $ f $ has support in $ [-\pi,\pi] $ and $ x/\sin(x/2) $ is bounded on compact subsets of $ (-2\pi,2\pi) $, it follows from \eqref{eq:section_preliminaries:variance_of_CUE} that there is a constant $ c > 0 $ such that 
 \begin{multline}\label{eq:section_proof_of_CUE:H_bound}
  \frac{1}{n^{1-\alpha}}\sum_{k=-\infty}^\infty\frac{\min(n,|k|)}{n^{1-\alpha}}\left|\F(f)\left(\frac{k}{n^{1-\alpha}}\right)\right|^2 = \V\left[X_n^{(\alpha,1)}(f)\right] \\
  \leq \frac{c}{2\pi^2}\int_\RR\int_\RR\left(\frac{f_\alpha(x)-f_\alpha(y)}{x-y}\right)^2\d x\d y = 2c\int_\RR|\xi||\F(f)(\xi)|^2\d\xi < \infty.
 \end{multline}
 By an approximation argument it follows that
 \begin{equation}
  \frac{1}{n^{1-\alpha}}\sum_{k=-\infty}^\infty\frac{\min(n,|k|)}{n^{1-\alpha}}\left|\F(f)\left(\frac{k}{n^{1-\alpha}}\right)\right|^2 \rightarrow \int_\RR|\xi||\F(f)(\xi)|^2\d\xi
 \end{equation}
 as $ n \rightarrow \infty $. By an analogous argument as for the tail of the first term, we get that the tail in the second term tends to zero as $ n \to \infty $.
\end{proof}

Note that  the proof of the latter result works also under the  weaker condition $ \eps = 0 $ in \eqref{eq:section_proof_of_CUE:epsilon_bound}. But later we  will use the stronger condition $\eps>0$. Note that this condition implies that $ \F(f) \in \L^1(\RR) $ and that $ f $ is H\"older continuous, which in general does not hold  in case we only have $ \eps = 0$.

\begin{lemma}\label{lem:section_proof_of_CUE:holder}
 Let $ f:\RR \rightarrow \RR $ fulfill the condition  \eqref{eq:section_proof_of_CUE:epsilon_bound} and let $ \delta < \frac{\eps}{2} $. Then $ f $ is H\"older continuous with exponent $ \delta $.  
\end{lemma}
\begin{proof}
 Let $ \theta_1,\theta_2 \in \RR $ then
 \begin{multline}
 |f(\theta_1)-f(\theta_2)| \leq \frac{1}{2\pi}\int_\RR|\F(f)(\xi)|\left|\e^{\i\theta_1\xi}-\e^{\i\theta_2\xi}\right|\d\xi \\
 \leq \frac{1}{\pi}\int_\RR\left|\F(f)(\xi)(1+|\xi|)^{\frac{1+\eps}{2}}\right|\frac{|\xi(\theta_1-\theta_2)|^{\delta}}{(1+|\xi|)^{\frac{1+\eps}{2}}}\d\xi \\
\leq |\theta_1-\theta_2|^\delta\frac{1}{\pi}\left(\int_\RR|\F(f)(\xi)|^2(1+|\xi|)^{1+\eps}\d\xi\int\right)^{1/2}\left(\int_\RR\frac{|\xi|^{2\delta}}{(1+|\xi|)^{1+\eps}}\d\xi\right)^{1/2},
\end{multline}
where we in the second inequality used that
\begin{equation}
 \left|\e^{\i\theta_1\xi}-\e^{\i\theta_2\xi}\right| = \left|\e^{\i\theta_1\xi}-\e^{\i\theta_2\xi}\right|^{1-\delta}\left|\e^{\i\theta_1\xi}-\e^{\i\theta_2\xi}\right|^\delta \leq 2|\xi(\theta_1-\theta_2)|^{\delta}.
\end{equation}

\end{proof}

\subsection{Riemann-Hilbert problem}\label{subsection:section_proof_of_CUE:riemann-hilbert_problem}
With the approximation \eqref{eq:section_proof_of_CUE:def_of_approximated_function} in hand, we continue to the steepest descent analysis of the following  Riemann-Hilbert problem.

\begin{RHP}\label{RHP:section_proof_of_CUE:first_problem}
Let $ \Gamma_m $ be the unit circle with positive orientation. Find a function $ m:\Gamma_m \rightarrow \CC^{2\times 2} $ such that
\begin{itemize}
 \item $ m $ is analytic in $ \CC \backslash \Gamma_m $
 \item $ m_+(z) = m_-(z)
 \begin{pmatrix}
\varphi_{n,t}(z) & -(\varphi_{n,t}(z) - 1)z^n \\
(\varphi_{n,t}(z) - 1)z^{-n} & 2 - \varphi_{n,t}(z)
\end{pmatrix}, $ $ z \in \Gamma_m $ 
 \item $ m(z) = I + \Ordo(z^{-1}) $ as $ |z| \rightarrow \infty $
\end{itemize}
where
\begin{equation}\label{eq:section_proof_of_CUE:def_of_varphi}
\varphi_{n,t}(z) = 1 - t\gamma_n\left(1-\e^{\lambda n^{-s}f_{\alpha,N}(z)}\right).
\end{equation}
\end{RHP}

Recall from Lemma \ref{lem:section_preliminaries:characteristic_function_as_integral} that we consider $ t \in [0,1] $.

The first step in the steepest descent analysis is the opening of the lens.  Usually, e.g. \cite{Deift99}, the location of the lips of the lens does not depend on $n$, but because of our special form of the $\varphi_{n,t}$ (coming from the mesoscopic scales) we need  to shrink the lenses as $ n \to \infty$ in a very specific way. This also means that we need to carefully check that all estimates in the error analysis still hold. 

\begin{defn}[$m \rightarrow S $]\label{def:section_proof_of_CUE:S_function}
Let $ \rho_N = \rho^{N^{-1}} $ for some $ \rho \in (0,1) $ and let $ \Gamma_S = \{z \in \CC: |z| \in\{ \rho_N,1,\rho_N^{-1}\} \} $. Define $ S: \CC \backslash \Gamma_S \rightarrow \CC^{2 \times 2} $ as
\begin{align}
S(z) =& m(z), & |z| & < \rho_N \\
S(z) =& m(z)
\begin{pmatrix}
1 & -(1 - \varphi_{n,t}(z)^{-1})z^n \\
0 & 1
\end{pmatrix}^{-1},
& \rho_N < |z| &< 1 \\
S(z) = & m(z)
\begin{pmatrix}
1 & 0 \\
(1 - \varphi_{n,t}(z)^{-1})z^{-n} & 1
\end{pmatrix},
& 1 < |z| & < \rho_N^{-1} \\
S(z) = & m(z), & \rho_{N}^{-1} < |z|&.
\end{align}
\end{defn}
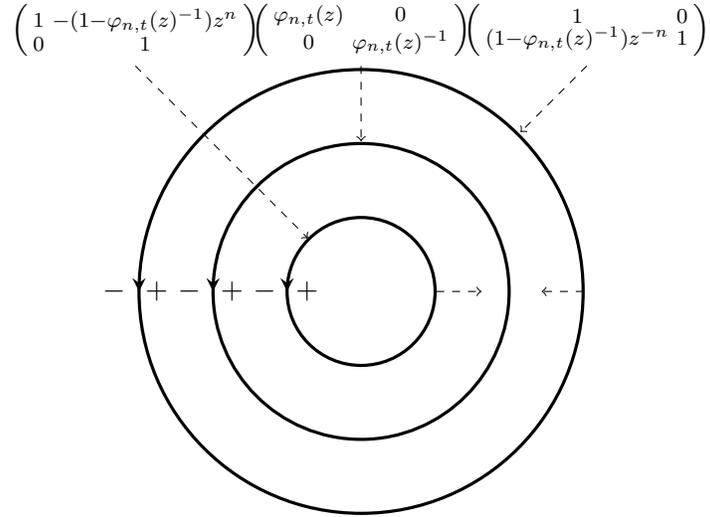
\begin{figure}[t]
\begin{center}
\begin{tikzpicture}[scale=4]
\begin{scope}[decoration={markings,mark= at position 0.5 with {\arrow{stealth}}}]
\draw[very thick,postaction=decorate] (0,0) circle(14pt);
\draw (-0.5,0) node[right] {$+$};
\draw (-0.5,0) node[left] {$-$};
\draw[very thick,postaction=decorate] (0,0) circle(7pt);
\draw (-0.25,0) node[right] {$+$};
\draw (-0.25,0) node[left] {$-$};
\draw[very thick,postaction=decorate] (0,0) circle(21pt);
\draw (-0.75,0) node[right] {$+$};
\draw (-0.75,0) node[left] {$-$};

\draw (-0.75,0.75) node[above] {$
\begin{psmallmatrix}
1 & -(1 - \varphi_{n,t}(z)^{-1})z^n \\
0 & 1
\end{psmallmatrix}
$};
\draw[->,dashed] (-0.75,0.75) --(-0.176777,0.176777);
\draw (0,0.75) node[above] {$
\begin{psmallmatrix}
\varphi_{n,t}(z) & 0 \\
0 & \varphi_{n,t}(z)^{-1}
\end{psmallmatrix}
$};
\draw[->,dashed] (0,0.75) --(0,0.5);
\draw (0.75,0.75) node[above] {$
\begin{psmallmatrix}
1 & 0 \\
(1 - \varphi_{n,t}(z)^{-1})z^{-n} & 1
\end{psmallmatrix}
$};
\draw[->,dashed] (0.75,0.75) --(0.530330,0.530330);
\draw[->,dashed] (0.25,0) --(0.4,0);
\draw[->,dashed] (0.74,0) --(0.6,0);
\end{scope}
\end{tikzpicture}
\end{center}
\caption{The contour and jump matrix for the Riemann-Hilbert problem stated in Lemma \ref{lem:section_proof_of_CUE:RHP_s_to_r}.}
\end{figure}

\begin{lemma}\label{lem:section_proof_of_CUE:RHP_s_to_r}
Let $ S $ be defined as in Definition \ref{def:section_proof_of_CUE:S_function}, then $ S $ solves the Riemann-Hilbert problem
\begin{itemize}
 \item $ S $ is analytic in $ \CC \backslash \Gamma_S $
 \item $ S_+(z)=S_-(z)J_S(z) $ where
\begin{align}
J_S(z) =&
\begin{pmatrix}
1 & -(1 - \varphi_{n,t}(z)^{-1})z^n \\
0 & 1
\end{pmatrix},
& |z| = \rho_N  \\
J_S(z) = &
\begin{pmatrix}
\varphi_{n,t}(z) & 0 \\
0 & \varphi_{n,t}(z)^{-1}
\end{pmatrix}, 
& |z| = 1 \\
J_S(z) = &
\begin{pmatrix}
1 & 0 \\
(1 - \varphi_{n,t}(z)^{-1})z^{-n} & 1
\end{pmatrix},
&|z| = \rho_N^{-1}
\end{align}
\item $ S(z) = I + \Ordo(z^{-1}) $ as $ z $ tends to infinity.
\end{itemize}
\end{lemma}
\begin{proof}
That $ S $ is well defined and analytic in $ \CC \backslash \Gamma_S $ follows if $ \varphi_{n,t} $ is non-zero for $ \rho_N \leq |z| \leq \rho_N^{-1} $. By the choice of $ \rho_N $ there is a uniform bound of $ f_{\alpha,N} $, for $ |z| = \rho_N $ or $ |z| = \rho_N^{-1} $, namely 
\begin{align}\label{eq:section_proof_of_CUE:bound_on_gn}
 |f_{\alpha,N}(z)| & \leq \frac{1}{n^{1-\alpha}} \sum_{k=-N}^N \left|\F(f) \left(\frac{k}{n^{1-\alpha}}\right)\right|\rho_N^{-N} \\
 & \rightarrow \rho^{-1} \int_{-\infty}^{\infty}|\F(f)(x)|\d x \\
 & \leq c\int_\RR|\F(f)(\xi)|^2(1+|\xi|)^{1+\epsilon} \d\xi < \infty
\end{align}
as $ n \rightarrow \infty $ for some constant $ c $. The limit above follows since the derivative of $ \F(f) $ is bounded and by the specific choice of $ N $ \eqref{eq:section_proof_of_CUE:def_N}. By the maximum modulo principle $f_{\alpha,N}(z)$ is bounded by the same constant for $ \rho_N \leq |z| \leq \rho_N^{-1} $. Note that this bound fails  if we did not shrink the lips. Since we have a uniform bound for $f_{\alpha,N}$, we let $ \lambda $ be so small that \eqref{eq:section_proof_of_CUE:def_of_varphi} is close to one for all $ z $ with $ \rho_N \leq |z| \leq \rho_N^{-1} $ which makes sure that all factors are analytic.

The jumps for $ |z| = \rho_N $ and $ |z| = \rho_N^{-1} $ follows by the definition of $ S $. The jump for $ |z| = 1 $ follows by the factorization
\begin{multline}
\begin{pmatrix}
\varphi_{n,t}(z) & -(\varphi_{n,t}(z) - 1)z^n \\
(\varphi_{n,t}(z) - 1)z^{-n} & 2 - \varphi_{n,t}(z)
\end{pmatrix}
= \\
=
\begin{pmatrix}
1 & 0 \\
(1 - \varphi_{n,t}(z)^{-1})z^{-n} & 1
\end{pmatrix}
\begin{pmatrix}
\varphi_{n,t}(z) & 0 \\
0 & \varphi_{n,t}(z)^{-1}
\end{pmatrix}
\begin{pmatrix}
1 & -(1 - \varphi_{n,t}(z)^{-1})z^n \\
0 & 1
\end{pmatrix}.
\end{multline}
This finishes the proof. 
\end{proof}

The idea with the first transformation is that the jumps $J_S$ on the inner and outer circle are close to the identity. If we ignore these jumps and only consider the  jump on the unit circle  we obtain a  Riemann-Hilbert problem that we can solve explicitly for finite $n$. The solution is called the global parametrix.
\begin{defn}[Global parametrix]\label{def:section_proof_of_CUE:P_function}
 Let $ \Gamma_P = \UC $ and define $ P:\CC \backslash \Gamma_P \rightarrow \CC^{2\times 2} $ as
\begin{equation}
 P(z) =
\begin{pmatrix}
e^{C(\log \varphi_{n,t})(z)} & 0 \\
 0 & e^{-C(\log \varphi_{n,t})(z)} 
\end{pmatrix}
\end{equation}
where $ C $ is the Cauchy operator on the unit circle defined for $ h\in \L^2(\UC) $ and given by
\begin{equation}\label{eq:section_proof_of_CUE:cauchy_operator_circle_plus}
 Ch(z) = \sum_{k=0}^{\infty}\hat{h}(k)z^k
\end{equation}
if $ |z| < 1 $ and
\begin{equation}\label{eq:section_proof_of_CUE:cauchy_operator_circle_minus}
 Ch(z) = -\sum_{k=-\infty}^{-1}\hat{h}(k)z^k
\end{equation}
if $ |z| > 1 $, and $ \log(1+z) $ is defined with  the principal branch. 
\end{defn}
\begin{rem}
In the proof of Lemma \ref{lem:section_proof_of_CUE:RHP_s_to_r} we saw that $ f_{\alpha,N} $ is uniformly bounded on $ \UC $, so by choosing $ \lambda $ small, $ \log \varphi_{n,t} =\log (1+ (\varphi_{n,t}-1)) $ is well-defined.
\end{rem}
The global parametrix solves the following Riemann-Hilbert problem. 
\begin{lemma}\label{lem:section_proof_of_CUE:RHP_global_parametrix}
Let $ P $ be defined as in Definition \ref{def:section_proof_of_CUE:P_function}, then $ P $ solves the Riemann-Hilbert problem
\begin{itemize}
\item $ P $ is analytic in $ \CC\backslash \Gamma_P $
\item $ P_+(z) = P_-(z)J_P(z) $ where 
\begin{equation}
J_P(z) = 
\begin{pmatrix}
\varphi_{n,t}(z) & 0 \\
0 & \varphi_{n,t}(z)^{-1}
\end{pmatrix},\quad z \in \Gamma_P
\end{equation}
\item $ P(z) = I + \Ordo(z^{-1}) $ as $ z $ tends to infinity.
\end{itemize}
\end{lemma}
\begin{proof}
 This is direct by our definition of the Cauchy operator on the circle.
\end{proof}

We now show that $P$ is a good approximation to $S$. 
\begin{defn}[$ S\rightarrow R $]\label{def:section_proof_of_CUE:R_function}
Let $ \Gamma_R = \Gamma_S $ and define $ R:\CC \backslash\Gamma_R \rightarrow \CC^{2\times 2} $ as
\begin{equation}
R(z) = S(z)P(z)^{-1}.
\end{equation}
\end{defn}
\begin{lemma}\label{lem:section_proof_of_CUE:solution_R_function}
Let $ R $ be defined as in Definition \ref{def:section_proof_of_CUE:R_function}, then $ R $ solves the Riemann-Hilbert problem
\begin{itemize}
 \item $ R $ is analytic in $ \CC\backslash \Gamma_R $
 \item $ R_+(z) = R_-(z)J_R(z) $ for $ z \in \Gamma_R $ where
 \begin{equation}
  J_R(z) = P_-(z)J_S(z)P_+(z)^{-1}
 \end{equation}
 \item $ R(z) = I + \Ordo(z^{-1}) $ as $ z $ tends to infinity.
\end{itemize}
Moreover, the solution exists for big enough $ n $ and
\begin{equation}
|(R(z)-I)|  \leq \frac{c\e^{\frac{n}{N}\log\rho}}{\min\{\rho_N - |z|,\rho_N^{-1} - |z|\}}
\end{equation}	
as $ n \rightarrow \infty $ for some positive constant $ c $.
\end{lemma}
\begin{proof}
That $ R $ solves the Riemann-Hilbert problem follows directly from the definition of $ R $.

To see that $ R $ is close to the identity requires more work. The structure of the proof is not different from the standard way, see e.g. \cite{Deift00} Theorem 7.171, but it requires some extra care since both the contour and the global parametrix depend on $ n $.

It is clear that $ J_R = I $ on the unit circle. If $ |z| = \rho_N $, then
\begin{equation}\label{eq:section_proof_of_CUE:error_term_on_small_circle}
J_R(z) = I-(1- \varphi_{n,t}(z)^{-1})z^ne^{2C(\log \varphi_{n,t})(z)} 
\begin{pmatrix}
0 & 1\\
0 & 0
\end{pmatrix}
\end{equation}
and similarly if $ |z| = \rho_N^{-1} $ then
\begin{equation}\label{eq:section_proof_of_CUE:error_term_on_big_circle}
J_R(z) = I + (1- \varphi_{n,t}(z)^{-1})z^{-n}e^{-2C(\log \varphi_{n,t})(z)} 
\begin{pmatrix}
0 & 0\\
1 & 0
\end{pmatrix}.
\end{equation}
We want to show that $ J_R $ is close to the identity, that is, that $ z^n $ and $ z^{-n} $ is the dominating factor on the inner, respectively outer circle of $ J_R $. That $ 1-\varphi_{n,t}^{-1} $ is uniformly bounded on $ \Gamma_R $ follows since we shrink the circles and hence have \eqref{eq:section_proof_of_CUE:bound_on_gn}. That we shrink the circles is crucial for the analysis to work. It does have the drawback that it complicates the bound for $ C\log\varphi_{n,t} $, but we will now show how to get around that. 

We recall that for $h(z)=\sum_{k=-\infty}^\infty \hat h(k)z^k$ the Wiener norm is defined by  $\|h(z)\|_W= \sum_{k=-\infty}^\infty| \hat h(k)|$. Then $ f_{\alpha,N} $  has finite Wiener norm 
\begin{align}\label{eq:section_proof_of_CUE:wiener_norm}
 \sum_{k=-\infty}^{\infty}|\hat{f}_{\alpha,N}(k)| & = \frac{1}{n^{1-\alpha}}\sum_{k=-N}^N\left|\F(f)\left(\frac{k}{n^{1-\alpha}}\right)\right| \\
 & \rightarrow \int_\RR|\F(f)(\xi)|\d\xi \\
 & \leq c'\int_\RR|\F(f)(\xi)|^2(1+|\xi|)^{1+\epsilon}\d\xi < \infty
\end{align}
as $ n \rightarrow \infty $ for some constant $ c' $ (the limit follows by the same argument used for \eqref{eq:section_proof_of_CUE:bound_on_gn}). In the last inequality we used Cauchy-Schwarz.  Note that the bound of the Wiener norm of $ f_{\alpha,N} $ is uniform in $n$. By submultiplicativity of the Wiener norm, i.e. $\|gh\|_W\leq \|g\|_W \|h\|_W$, we obtain, by \eqref{eq:section_proof_of_CUE:cauchy_operator_circle_plus} and \eqref{eq:section_proof_of_CUE:cauchy_operator_circle_minus}, 
\begin{align}\label{eq:section_proof_of_CUE:bound_of_cauchy_operator_of_logarithm}
 |C(\log \varphi_{n,t})(z)| & \leq \sum_{k=-\infty}^\infty|\widehat{\log\varphi_{n,t}}(k)| \\
 & \leq -\log\left(1-|t\gamma_n|\left(\e^{|\lambda| n^{-s}\sum_{k=-\infty}^\infty|\hat{f}_{\alpha,N}(k)|}-1\right)\right) \\
 & \leq -\log\left(1-\left(\e^{|\lambda|c'\int_\RR|\F(f)(\xi)|^2(1+|\xi|)^{1+\epsilon}\d\xi}-1\right)\right)
\end{align}
for $ |z| = \rho_N,\rho_N^{-1} $. From this discussion and \eqref{eq:section_proof_of_CUE:error_term_on_small_circle} and \eqref{eq:section_proof_of_CUE:error_term_on_big_circle}
  there exists a constant $ c'' $ such that
 \begin{align}\label{eq:section_proof_of_CUE:order_of_jump}
 \|J_R - I\|_{\L^{\infty}(\Gamma_R)} & \leq c''\rho_N^n \\
 & = c''\e^{\frac{n}{N}\log\rho}.
\end{align}
as $ n \rightarrow \infty $. General principles for Riemann-Hilbert problems (see \cite{Deift00} Theorem 7.103) now tell us that for big enough $ n $ we can solve the Riemann-Hilbert problem  by a Neumann series. 

First recall the Cauchy operator on $ \Gamma_R $, it is defined for $ h \in \L^2(\Gamma_R) $ and given by
\begin{equation}
 C^{\Gamma_R}h(z) = \frac{1}{2\pi i}\int_{\Gamma_R}\frac{h(w)}{w-z}\d w.
\end{equation}
The operator $ C^{\Gamma_R}_- $ takes a function $ h $ to the limiting function of $ C^{\Gamma_R}h $ as $ z $ tends to $ \Gamma_R $ from the minus side. Using the Cauchy operator on $ \Gamma_R $, the solution to the Riemann-Hilbert problem is given by
\begin{equation}
R = I + (C^{\Gamma_R}(\mu w))(z)
\end{equation}
where $ w = J_R - I $ and $ \mu $ is a solution to $ (I-C^{\Gamma_R}_- w)\mu = I $. 

Since $ J_R - I = 0 $ on the unit circle and $ \mu \in \L^2(\Gamma_R) $, we have, from \eqref{eq:section_proof_of_CUE:order_of_jump}, for any  $ z \in \CC $ with $ \rho_N < |z| < \rho_N^{-1} $, that
\begin{align}
|R(z)-I| & \leq \frac{1}{2\pi}\int_{\Gamma_R \backslash \UC} \frac{|\mu(w)||(J_R - I)(w)|}{|w-z|} |\d w| \\
 & \leq \frac{(\rho_N+\rho_N^{-1})\|\mu\|_{\L^2(\Gamma_R)}}{\min\{\rho_N - |z|,\rho_N^{-1} - |z|\}}c'' \e^{\frac{n}{N}\log\rho}.
\end{align} 
Since $ w $ tends to zero as $ n $ tends to infinity, we may express $ \mu $ as a Neumann series and bound $ \|\mu\|_{\L^2(\Gamma_R)} $ with any constant larger than one for big enough $ n $.
\end{proof}

\subsection{Asymptotic behavior of the moment generating function}\label{subsection:section_proof_of_CUE:the_asymptotic_behavior}
In this section we will calculate the leading term of the moment generating function asymptotically, found from the Riemann-Hilbert analysis, and let $ n $ tend to infinity to prove the main results.

Since the transformations $ m $ to $ R $, done in the Riemann-Hilbert analysis, are reversible, Lemma \ref{lem:section_proof_of_CUE:solution_R_function} gives us the existence of $ m $ and by tracing back $ R\rightarrow S \rightarrow m $ we get the asymptotic behavior of $ m $. By \eqref{eq:section_preliminaries:capital_f_one} and \eqref{eq:section_preliminaries:capital_f_two},
\begin{align}\label{eq:section_proof_of_CUE:relation_between_f_and_g}
F^{(1)}(z) & = m_+(z)f^{(1)}(z) \\
& = R(z)G^{(1)}(z) \\
& = G^{(1)}(z) + (R(z)-I)G^{(1)}(z)
\end{align}
and 
\begin{equation}\label{eq:section_proof_of_CUE:relation_between_f2_and_g2}
F^{(2)}(z) = G^{(2)}(z) + \left((R_+(z)^{-1})^T-I\right)G^{(2)}(z)
\end{equation}
where $ G^{(1)} $ and $ G^{(2)} $ are the leading part of $ F^{(1)} $ and $ F^{(2)} $ that comes from the global parametrix, given by
\begin{align}\label{eq:section_proof_of_CUE:capital_g_one}
G^{(1)}(z) & = P_+(z)
\begin{pmatrix}
 1 & -(1-\varphi_{n,t}(z)^{-1})z^n \\
 0 & 1
\end{pmatrix}
f^{(1)}(z) \\
& = \frac{1}{2 \pi \i}\left(\e^{C_-(\log \varphi_{n,t})(z)}z^n, \e^{-C_+(\log \varphi_{n,t})(z)} \right)^T
\end{align}
and
\begin{equation}\label{eq:section_proof_of_CUE:capital_g_two}
G^{(2)}(z) =  \left(\e^{-C_+(\log \varphi_{n,t})(z)}z^{-n}, -\e^{C_-(\log \varphi_{n,t})(z)} \right)^T
\end{equation}
where $ C_+ $ and $ C_- $ are defined as \eqref{eq:section_proof_of_CUE:cauchy_operator_circle_plus} and \eqref{eq:section_proof_of_CUE:cauchy_operator_circle_minus} respectively but with $ z \in \UC $.

The following lemma is the result of the Riemann-Hilbert analysis, it gives us the leading term of the moment generating function.
\begin{lemma}\label{lem:section_proof_of_CUE:approximation_of_integral}
 Let $ F^{(1)} $, $ F^{(2)} $, $ G^{(1)} $ and $ G^{(2)} $ be defined as in \eqref{eq:section_preliminaries:capital_f_one}, \eqref{eq:section_preliminaries:capital_f_two}, \eqref{eq:section_proof_of_CUE:capital_g_one} and \eqref{eq:section_proof_of_CUE:capital_g_two} respectively. Then
\begin{multline}
 \int_0^1\int_{\UC}(\varphi_{n,1}(z) - 1)F^{(1)'}(z)^T F^{(2)}(z)\d x \d t = \\
 = \int_0^1\int_{\UC}(\varphi_{n,1}(z) - 1)G^{(1)'}(z)^T G^{(2)}(z) \d z \d t + \Ordo\left(\e^{\frac{n}{2N}\log\rho}\right)
\end{multline}
 as $ n $ tends to infinity.
\end{lemma}

\begin{proof}
To prove this we will in particular use \eqref{eq:section_proof_of_CUE:relation_between_f_and_g} and Lemma \ref{lem:section_proof_of_CUE:solution_R_function}, so a first step is to control $ G^{(1)} $. From \eqref{eq:section_proof_of_CUE:bound_of_cauchy_operator_of_logarithm} and \eqref{eq:section_proof_of_CUE:bound_on_gn} we see that $ C(\log \varphi_{n,t}) $ (or more precisely, the analytic extension of both $ C_+(\log\varphi_{n,t}) $ and $ C_-(\log\varphi_{n,t}) $ respectively) is uniformly bounded for $ \rho_N \leq |z| \leq \rho_N $.
So for $ 1\leq |z| \leq \rho^{-n^{-1}} $,
\begin{equation}
 |G^{(1)}(z)| \leq c'
\end{equation}
for some constant $ c' $, and by \eqref{eq:section_proof_of_CUE:relation_between_f_and_g} and Lemma \ref{lem:section_proof_of_CUE:solution_R_function}
\begin{equation}
 |F^{(1)}(z) - G^{(1)}(z)| \leq \frac{c'c}{\rho_N^{-1} - \rho^{-n^{-1}}}\e^{\frac{n}{N}\log\rho} 
\end{equation}
for $ |z| = \rho^{-n^{-1}} $. Similarly, we get the the same estimate for $ |z| = \rho^{n^{-1}} $ when considering the analytic continuation of $ F^{(1)} - G^{(1)} $. By Cauchy's integral formula,
\begin{equation}
 |F^{(1)'}(z) - G^{(1)'}(z)| \leq \frac{c''}{(\rho_N^{-1} - \rho^{-n^{-1}})(\rho^{-n^{-1}}-1)^2}\e^{\frac{n}{N}\log\rho}
\end{equation}
for $ z \in \UC $ and some constant $ c'' $. From \eqref{eq:section_proof_of_CUE:relation_between_f2_and_g2} we get a similar relation between $ F^{(2)} $ and $ G^{(2)} $. 

All this tells us that for big enough $ n $,
\begin{align}
& \left|\int_{\UC}(\varphi_{n,1}(z)-1) \left(F^{(1)'}(z)^TF^{(2)}(z) - G^{(1)'}(z)^TG^{(2)}(z)\right)\d z\right| \\
\leq & \left\|\left(F^{(1)'}\right)^TF^{(2)} - \left(G^{(1)'}\right)^TG^{(2)}\right\|_{\L^\infty(\UC)} \int_{\UC}|(\varphi_{n,1}(z)-1)||\d z| \\
\leq & c\e^{\frac{n}{2N}\log\rho}.
\end{align}
for some constant $ c>0 $.
\end{proof}

\begin{lemma}\label{lem:section_proof_of_CUE:approximated_characteristic_function}
Let $ G^{(1)} $ and $ G^{(2)} $ be defined as in \eqref{eq:section_proof_of_CUE:capital_g_one} and \eqref{eq:section_proof_of_CUE:capital_g_two} respectively, then
\begin{multline}
\int_0^1 \int_{\UC}(\varphi_{n,1}(z) - 1)G^{(1)'}(z)^T G^{(2)}(z) \d z \d t = \\
= n\widehat{\log\varphi}_{n,1}(0) + \sum_{k=1}^{\infty}k\widehat{\log\varphi}_{n,1}(k)\widehat{\log\varphi}_{n,1}(-k).
\end{multline}
\end{lemma}
\begin{proof}
This result is a variation to  a computation in  \cite[Ex. 3]{Deift99}, but for completeness we include a proof.

A straightforward calculation, using \eqref{eq:section_proof_of_CUE:capital_g_one}, \eqref{eq:section_proof_of_CUE:capital_g_two}  and $C_+-C_-=I$, shows
\begin{multline}\label{eq:section_proof_of_CUE:first_simplification_of_integral}
\int_0^1\int_{\UC}(\varphi_{n,1}(z) - 1)G^{(1)'}(z)^T G^{(2)}(z) \d z \d t =\\ 
= \int_0^1 \int_{\UC}\frac{\varphi_{n,1}(z)-1}{\varphi_{n,t}}\left(C_+(\log \varphi_{n,t})'(z) + C_-(\log \varphi_{n,t})'(z) + \frac{n}{z}\right)\frac{\d z}{2\pi i}\d t.
\end{multline}
Since by definition \eqref{eq:section_proof_of_CUE:def_of_varphi} of $\varphi_{n,t}$ we have
\begin{equation}
 \frac{\varphi_{n,1}(z)-1}{\varphi_{n,t}(z)}=\frac{\partial}{\partial t} \log\varphi_{n,t}(z),
\end{equation}
we  see that the last term in \eqref{eq:section_proof_of_CUE:first_simplification_of_integral} is equal to $ n\widehat{\log\varphi}_{n,1}(0) $. 
To compute the first part of \eqref{eq:section_proof_of_CUE:first_simplification_of_integral}, write the Cauchy operator as in \eqref{eq:section_proof_of_CUE:cauchy_operator_circle_plus}, then
\begin{align}
&\int_0^1 \int_{\UC}\frac{\varphi_{n,1}(z)-1}{\varphi_{n,t}(z)} \left(C_+(\log \varphi_{n,t})'(z) + C_-(\log \varphi_{n,t})'(z)\right)\frac{\d z}{2\pi i}\d t \\
= & \int_0^1\int_{\UC}\frac{\partial}{\partial t}(\log \varphi_{n,t}(z)) \sum_{k = 0}^{\infty}k\left(\widehat{\log\varphi}_{n,t}(k)z^{k-1} + \widehat{\log\varphi}_{n,t}(-k)z^{-k-1}\right)\frac{\d z}{2\pi i} \d t \\
= & \int_0^1\sum_{k = 0}^{\infty}k \left(\widehat{\log\varphi}_{n,t}(k)\frac{\partial}{\partial t}\left(\int_{\UC}\log \varphi_{n,t}(z)z^k\frac{\d z}{2\pi i z}\right)\right. \\
 & \left.+ \widehat{\log\varphi}_{n,t}(-k)\frac{\partial}{\partial t}\left(\int_{\UC}\log\varphi_{n,t}(z)z^{-k}\frac{\d z}{2\pi i z} \right)\right)\d t \\
= & \int_0^1\sum_{k=0}^{\infty}k \frac{\partial}{\partial t}\left(\widehat{\log\varphi}_{n,t}(k)\widehat{\log\varphi}_{n,t}(-k)\right)\d t \\
= & \sum_{k=0}^{\infty}k\widehat{\log\varphi}_{n,1}(k)\widehat{\log\varphi}_{n,1}(-k),
\end{align}
where the changes of limits are easily justified.
\end{proof}

In the next step we replace $\log \varphi_{n,1}$ by $\lambda f_{\alpha,N}$ in the second term in Lemma~ \ref{lem:section_proof_of_CUE:approximated_characteristic_function}. Here it is important to have a more precise choice of $N$. We choose $N$ such that \eqref{eq:section_proof_of_CUE:def_N} holds and 
\begin{equation}
 n^{-2\delta}\frac{N^{\frac{1}{2}}}{n^{\frac{1-\alpha}{2}}} \to 0,
\end{equation}
as $n \to \infty$. 
\begin{lemma} \label{lem:section_proof_of_CUE:varphitof} We have 
\begin{multline}
 \sum_{k=0}^{\infty}k\widehat{\log\varphi}_{n,1}(k)\widehat{\log\varphi}_{n,1}(-k) \label{eq:section_proof_of_CUE:general_sobolev_term}\\
 = \frac{\lambda^2 n^{-2s}}{n^{1-\alpha}}\sum_{k=1}^N\frac{k}{n^{1-\alpha}}\F(f)\left(\frac{k}{n^{1-\alpha}}\right)\F(f)\left(\frac{-k}{n^{1-\alpha}}\right) + \Ordo\left(n^{-2\delta}\frac{N^{\frac{1}{2}}}{n^{\frac{1-\alpha}{2}}}\right),
\end{multline}
as $n \to \infty$. 
\end{lemma}
\begin{proof}
 Write
\begin{equation}\label{eq:section_proof_of_CUE:first_expansion_of_varphi}
 \log \varphi_{n,1} = \lambda n^{-s}f_{\alpha,N}+\log\left(1+\kappa n^{-\delta}\left(\e^{-\lambda n^{-s}f_{\alpha,N}}-1\right)\right).
\end{equation}
Using Cauchy-Schwartz inequality repeatedly yields 
\begin{align}
 & \left|\sum_{k=0}^{\infty}k\widehat{\log\varphi}_{n,1}(k)\widehat{\log\varphi}_{n,1}(-k) - \sum_{k=0}^{\infty}k\lambda n^{-s}\hat{f}_{\alpha,N}(k)\lambda n^{-s}\hat{f}_{\alpha,N}(-k)\right| \\
 \leq &  \frac{1}{2}\sum_{k=-\infty}^{\infty}|k||\widehat{\log\varphi}_{n,1}(k)||\widehat{\log\varphi}_{n,1}(-k)-\lambda n^{-s}\hat{f}_{\alpha,N}(-k)| \\
 & + \frac{1}{2}\sum_{k=-\infty}^{\infty}|k||\widehat{\log\varphi}_{n,1}(k)-\lambda n^{-s}\hat{f}_{\alpha,N}(k)||\lambda n^{-s}\hat{f}_{\alpha,N}(-k)| \\
  \leq &  \frac{1}{2}\left(\sum_{k=-\infty}^{\infty}|k||\widehat{\log\varphi}_{n,1}(k)|^2\right)^{\frac{1}{2}}\left(\sum_{k=-\infty}^{\infty}|k||\widehat{\log\varphi}_{n,1}(k)-\lambda n^{-s}\hat{f}_{\alpha,N}(k)|^2\right)^{\frac{1}{2}} \\
 & + \frac{1}{2}\left(\sum_{k=-\infty}^{\infty}|k||\widehat{\log\varphi}_{n,1}(k)-\lambda n^{-s}\hat{f}_{\alpha,N}(k)|^2\right)^{\frac{1}{2}}\left(\sum_{k=-\infty}^{\infty}|k||\lambda n^{-s}\hat{f}_{\alpha,N}(k)|^2\right)^{\frac{1}{2}} \\
 \leq & \frac{1}{2}\left(\left(\sum_{k=-\infty}^{\infty}|k||\widehat{\log\varphi}_{n,1}(k)|^2\right)^{\frac{1}{2}}+\left(\sum_{k=-\infty}^{\infty}|k||\lambda n^{-s}\hat{f}_{\alpha,N}(k)|^2\right)^{\frac{1}{2}}\right) \\
 & \times \|\frac{\d}{\d\theta}(\log\varphi_{n,1}-\lambda n^{-s}f_{\alpha,N})\|_{\L^2(\UC)}\|\log\varphi_{n,1}-\lambda n^{-s}f_{\alpha,N}\|_{\L^2(\UC)}.
\end{align}
By \eqref{eq:section_proof_of_CUE:H_bound} it is clear that we need to show that the last factor tends to zero. For that we use \eqref{eq:section_proof_of_CUE:first_expansion_of_varphi} and the fact that $ f_{\alpha,N} $ is uniformly bounded on $ \UC $, to see that
\begin{align}
 & \|\frac{\d}{\d\theta}(\log\varphi_{n,1}-\lambda n^{-s}f_{\alpha,N})\|_{\L^2(\UC)}\|\log\varphi_{n,1}-\lambda n^{-s}f_{\alpha,N}\|_{\L^2(\UC)} \\
 \leq & c n^{-2\delta}\|\frac{\d}{\d\theta}f_{\alpha,N}\|_{\L^2(\UC)}\|f_{\alpha,N}\|_{\L^2(\UC)} \\
 \leq & c n^{-2\delta}\frac{N^{\frac{1}{2}}}{n^{\frac{1-\alpha}{2}}}\left(\frac{1}{n^{1-\alpha}}\sum_{k=-N}^N\frac{|k|}{n^{1-\alpha}}\left|\F(f)\left(\frac{k}{n^{1-\alpha}}\right)\right|^2\right)^{\frac{1}{2}} \\
 &\times\left(\frac{1}{n^{1-\alpha}}\sum_{k=-N}^N\left|\F(f)\left(\frac{k}{n^{1-\alpha}}\right)\right|^2\right)^{\frac{1}{2}}
\end{align}
where $ N^{\frac{1}{2}} $ comes from the $ \L^2 $-norm of the derivative and $ n^{-\frac{1-\alpha}{2}} $ comes from the $ \L^2 $-norm of the function, and $ c $ is some constant. By \eqref{eq:section_proof_of_CUE:H_bound} and \eqref{eq:section_proof_of_CUE:L2_bound} this behaves as
\begin{equation}\label{eq:section_proof_of_CUE:order_of_H_norm_error}
 n^{-2\delta}\frac{N^{\frac{1}{2}}}{n^{\frac{1-\alpha}{2}}}
\end{equation}
as $ n \rightarrow \infty $. Hence by \eqref{eq:section_proof_of_CUE:order_of_H_norm_error},  \eqref{eq:section_proof_of_CUE:def_of_approximated_function} and \eqref{eq:section_proof_of_CUE:fourier_coefficient_in_terms_of_fourier_transform}
we obtain the statement. 
\end{proof}

We need a final estimate.
\begin{lemma}
We have 
\begin{equation}\label{eq:section_proof_of_CUE:order_of_l1}
 \int_{-\pi}^{\pi}|f_{\alpha,N}(\e^{\i\theta})-f_{\alpha}(\theta)|\d\theta = o(n^{\alpha-1}).
\end{equation}
as $n \to \infty$. 
\end{lemma}

\begin{proof}
Note that for $ \alpha = 1 $, this would be a trivial  statement, since $ \|\cdot \|_{\L^1(\UC)} \leq \sqrt{2\pi}\| \cdot \|_{\L^2(\UC)} $ and $f_{1,N}$ is a truncated Fourier series of $f_1$ (with the natural identification of the interval $[-\pi,\pi)$ with the unit circle).  But for $0<\alpha<1$ the function $f_\alpha$ also depends on $n$ and the statement is not immediate.

Note that by \eqref{eq:section_proof_of_CUE:def_of_approximated_function}, \eqref{eq:section_proof_of_CUE:fourier_coefficient_in_terms_of_fourier_transform}, and a change of variable  we can write
\begin{multline}\label{eq:section_proof_of_CUE:l1_error}
 n^{1-\alpha}\int_{-\pi}^\pi\left|f_{\alpha,N}(\e^{\i\theta})-f_{\alpha}(\theta)\right|\d\theta \\
 = \int_{-\pi n^{1-\alpha}}^{\pi n^{1-\alpha}}\left|\frac{1}{n^{1-\alpha}} \sum_{k=-N}^N\F(f)\left(\frac{k}{n^{1-\alpha}}\right)\e^{\i\frac{k}{n^{1-\alpha}}} - f(\theta)\right|\d\theta.
 \end{multline}
The sum in the integrand is a Riemann sum for the inverse Fourier transform of the Fourier transform of $f$ and, thus, we expect it to converge to $f$. We will show that it does in the sense that  \eqref{eq:section_proof_of_CUE:l1_error} is $o(1)$ as $n \to \infty$, from which \eqref{eq:section_proof_of_CUE:order_of_l1} follows. 
First, we compare  the sum in the integrand to the truncated inverse Fourier integral of the Fourier transform of $f$. 
Then,  by the triangle inequality, we obtain
 \begin{multline}\label{eq:section_proof_of_CUE:l1_error2}
  n^{1-\alpha}\int_{-\pi}^\pi\left|f_{\alpha,N}(\e^{\i\theta})-f_{\alpha}(\theta)\right|\d\theta \\
 \leq \int_{-\pi n^{1-\alpha}}^{\pi n^{1-\alpha}}\left|\frac{1}{n^{1-\alpha}} \sum_{k=-N}^N\F(f)\left(\frac{k}{n^{1-\alpha}}\right)\e^{\i\frac{k}{n^{1-\alpha}}} - \int_{-\frac{N+1/2}{n^{1-\alpha}}}^{\frac{N+1/2}{n^{1-\alpha}}}\F(f)(\xi)\e^{\i\xi\theta}\d\xi\right|\d\theta \\
 + \int_{-\pi n^{1-\alpha}}^{\pi n^{1-\alpha}}\left|\int_{-\frac{N+1/2}{n^{1-\alpha}}}^{\frac{N+1/2}{n^{1-\alpha}}}\F(f)(\xi)\e^{\i\xi\theta}\d\xi-f(\theta)\right|\d\theta. 
\end{multline}

We show that the first term  at the right-hand side of \eqref{eq:section_proof_of_CUE:l1_error2} tends to zero by comparing  the Dirichlet kernel with the sine kernel. Indeed, we recall that $ \sum_{k=-N}^N\e^{-\i k x} = \frac{\sin((N+1/2)x)}{\sin(x/2)} $ and, by changing the order of integration and a change of variables, we rewrite the integral to
\begin{multline}\label{eq:section_proof_of_CUE:l1error3}
 \int_{-\pi n^{1-\alpha}}^{\pi n^{1-\alpha}}\left|\int_{-\pi}^\pi f(x)\frac{1}{n^{1-\alpha}} \left(\frac{\sin\frac{N+1/2}{n^{1-\alpha}}(x-\theta)}{\sin\frac{x-\theta}{2n^{1-\alpha}}} - \frac{\sin\frac{N+1/2}{n^{1-\alpha}}(x-\theta)}{\frac{x-\theta}{2n^{1-\alpha}}}\right)\right|\d\theta \\
 = \int_{-\pi}^\pi\left|\int_{-\pi}^\pi f(x)\sin\left((N+1/2)\left(\frac{x}{n^{1-\alpha}}-y\right)\right)g_n(x,y)\d x\right|\d y,
\end{multline}
where 
\begin{equation}
 g_n(x,y) = \frac{1}{\sin\frac{1}{2}\left(\frac{x}{n^{1-\alpha}}-y\right)} - \frac{1}{\frac{1}{2}\left(\frac{x}{n^{1-\alpha}}-y\right)}.
\end{equation}
Note that the integrand is bounded on $ [-\pi,\pi]\times[-\pi,\pi] $. For a fixed $ y \in [-\pi,\pi] $,
\begin{multline}
 \int_{-\pi}^\pi f(x)\sin\left((N+1/2)\left(\frac{x}{n^{1-\alpha}}-y\right)\right)g_n(x,y)\d x \\
 = \int_{-\pi}^\pi f(x)\sin\left((N+1/2)\left(\frac{x}{n^{1-\alpha}}-y\right)\right)g_n(0,y)\d x \\
 + \int_{-\pi}^\pi f(x)\sin\left((N+1/2)\left(\frac{x}{n^{1-\alpha}}-y\right)\right)\left(g_n(x,y)-g_n(0,y)\right)\d x,
\end{multline}
where, since $ g_n(0,y) $ is independent of $ n $ and $ (N+1/2)/n^{1-\alpha} \to \infty $, the first term tends to zero by Riemann-Lebesgue Lemma, and the second term tends to zero since $ g_n(x,y) \to g_n(0,y) $ uniformly, as $ n \to 0 $. The integrand of the outer integral at  the right-hand side of \eqref{eq:section_proof_of_CUE:l1error3}  converges to zero for every $y \in [-\pi,\pi)$ and  is bounded. Hence it tends to zero as $n \to \infty$ by Lebesgues dominated convergence theorem and  this  implies that the first term in \eqref{eq:section_proof_of_CUE:l1_error2} tends to zero as $n \to \infty$.

Now we come to the second term in \eqref{eq:section_proof_of_CUE:l1_error2}.  First note that 
\begin{equation}
 \int_{-\frac{N+1/2}{n^{1-\alpha}}}^{\frac{N+1/2}{n^{1-\alpha}}}\F(f)(\xi)\e^{\i\xi\theta}\d\xi \to f(\theta),
\end{equation}
in $ \L^2 $-sense. This is however not enough to claim convergence of the second term at the right-hand side of \eqref{eq:section_proof_of_CUE:l1_error2}, since we need an $\mathbb L_1$-type convergence. We will show this by controlling the tails of $\mathcal F (f)$ by using the fact that $f$ is H\"older continuous  by Lemma \ref{lem:section_proof_of_CUE:holder}. 

We divide the interval of  integration  into two parts, one which is bounded in $ n $ and contains the support of $ f $, and the complement. More precisely,
\begin{multline}\label{eq:section_proof_of_CUE:second_error}
 \int_{-\pi n^{1-\alpha}}^{\pi n^{1-\alpha}}\left|\int_{-\frac{N+1/2}{n^{1-\alpha}}}^{\frac{N+1/2}{n^{1-\alpha}}}\F(f)(\xi)\e^{\i\xi\theta}\d\xi-f(\theta)\right|\d\theta \\
 = \int_{-2\pi}^{2\pi}\left|\int_{-\frac{N+1/2}{n^{1-\alpha}}}^{\frac{N+1/2}{n^{1-\alpha}}}\F(f)(\xi)\e^{\i\xi\theta}\d\xi-f(\theta)\right|\d\theta \\+ \int_{\pi n^{1-\alpha} \geq |\theta|\geq 2\pi}\left|\int_{-\frac{N+1/2}{n^{1-\alpha}}}^{\frac{N+1/2}{n^{1-\alpha}}}\F(f)(\xi)\e^{\i\xi\theta}\d\xi\right|\d\theta.
\end{multline}
 For the second term,  we first rewrite the integral  by  changing the order of integration (recall that the support of $ f $ is in $ [-\pi,\pi] $)
\begin{multline}
 \int_{\pi n^{1-\alpha} \geq |\theta|\geq 2\pi}\left|\int_{-\frac{N+1/2}{n^{1-\alpha}}}^{\frac{N+1/2}{n^{1-\alpha}}}\F(f)(\xi)\e^{\i\xi\theta}\d\xi\right|\d\theta \\
= \int_{\pi n^{1-\alpha} \geq |\theta|\geq 2\pi}\left|\int_{-\pi}^\pi f(\omega)\int_{-\frac{N+1/2}{n^{1-\alpha}}}^{\frac{N+1/2}{n^{1-\alpha}}}\e^{-\i\xi(\omega-\theta)}\d\xi\d\omega\right|\d\theta \\
 = 2\int_{\pi n^{1-\alpha} \geq |\theta|\geq 2\pi}\left|\int_{-\pi}^\pi \frac{f(\omega)}{\omega-\theta}\sin\left(\frac{N+1/2}{n^{1-\alpha}}(\omega-\theta)\right)\d\omega\right|\d\theta. \\
\end{multline}
By the basic trigonometric addition formula for sine the latter is bounded by 
\begin{multline}
 2\int_{\pi n^{1-\alpha} \geq |\theta|\geq 2\pi}\frac{1}{|\theta|}\left|\int_{-\pi}^\pi \frac{f(\omega)}{1-\frac{\omega}{\theta}}\sin\left(\frac{N+1/2}{n^{1-\alpha}}\omega\right)\d\omega\right|\d\theta \\
 + 2\int_{\pi n^{1-\alpha} \geq |\theta|\geq 2\pi}\frac{1}{|\theta|}\left|\int_{-\pi}^\pi \frac{f(\omega)}{1-\frac{\omega}{\theta}}\cos\left(\frac{N+1/2}{n^{1-\alpha}}\omega\right)\d\omega\right|\d\theta.
\end{multline}
The integral over $ \theta $ grows logarithmically in $ n $ if we ignore the  integral over $\omega$. Therefore  we proceed by  computing the rate of convergence of 
\begin{equation}\label{eq:section_proof_of_CUE:fourier_coefficient}
 \int_{-\pi}^\pi \frac{f(\omega)}{1-\frac{\omega}{\theta}}\e^{{\rm i}\frac{N+1/2}{n^{1-\alpha}}\omega}\d\omega. 
\end{equation}
 Since $ \e^{-\i\pi} = -1 $ and since $ f $ has support in $ [-\pi,\pi] $, 
\begin{multline}\label{eq:section_proof_of_CUE:last}
 \int_{-\pi}^\pi \frac{f(\omega)}{1-\frac{\omega}{\theta}}\e^{\i\frac{N+1/2}{n^{1-\alpha}}\omega}\d\omega = - \int_{-\pi}^\pi \frac{f(\omega)}{1-\frac{\omega}{\theta}}\e^{\i\frac{N+1/2}{n^{1-\alpha}}\left(\omega-\frac{n^{1-\alpha}}{N+1/2}\pi\right)}\d\omega \\
 = -\int_{-\pi}^\pi \frac{f\left(\omega+\frac{n^{1-\alpha}}{N+1/2}\pi\right)}{1-\frac{\omega+\frac{n^{1-\alpha}}{N+1/2}\pi}{\theta}}\e^{\i\frac{N+1/2}{n^{1-\alpha}}\omega}\d\omega - \int_{-(1+\frac{n^{1-\alpha}}{N+1/2})\pi}^{-\pi} \frac{f\left(\omega+\frac{n^{1-\alpha}}{N+1/2}\pi\right)}{1-\frac{\omega+\frac{n^{1-\alpha}}{N+1/2}\pi}{\theta}}\e^{\i\frac{N+1/2}{n^{1-\alpha}}\omega}\d\omega.
\end{multline}
By writing \eqref{eq:section_proof_of_CUE:fourier_coefficient} as the average of itself and the right-hand side of \eqref{eq:section_proof_of_CUE:last}, we get
\begin{multline}\label{eq:section_proof_of_CUE:riemann-lebegue_holder_function}
 \left|\int_{-\pi}^\pi \frac{f(\omega)}{1-\frac{\omega}{\theta}}\e^{i\frac{N+1/2}{n^{1-\alpha}}\omega}\d\omega\right| \\
 \leq \frac{1}{2}\left|\int_{-\pi}^\pi \left(\frac{f(\omega)}{1-\frac{\omega}{\theta}}-\frac{f\left(\omega+\frac{n^{1-\alpha}}{N+1/2}\pi\right)}{1-\frac{\omega+\frac{n^{1-\alpha}}{N+1/2}\pi}{\theta}}\right)\e^{i\frac{N+1/2}{n^{1-\alpha}}\omega}\d\omega\right| + \Ordo\left(\frac{n^{1-\alpha}}{N+1/2}\right),
\end{multline}
where  we also used that $ \frac{1}{4}\leq 1-\frac{\omega+\frac{n^{1-\alpha}}{N+1/2}\pi}{\theta} \leq \frac{7}{4} $ to bound the last integral at the right-hand side of \eqref{eq:section_proof_of_CUE:last}. Since $ f $ is H\"older continuous with exponent $ \delta < \frac{\eps}{2} $, by Lemma \ref{lem:section_proof_of_CUE:holder}, and since $ \frac{1}{2}\leq 1-\frac{\omega}{\theta} \leq \frac{3}{2} $ we get that \eqref{eq:section_proof_of_CUE:riemann-lebegue_holder_function} is bounded by
\begin{equation}
 \frac{1}{2}\int_{-\pi}^\pi \frac{|f(\omega)-f(\omega+\frac{n^{1-\alpha}}{N+1/2}\pi)|}{\left|1-\frac{\omega+\frac{n^{1-\alpha}}{N+1/2}\pi}{\theta}\right|}\d\omega + \Ordo\left(\frac{n^{1-\alpha}}{N+1/2}\right)
 \leq c\left(\frac{n^{1-\alpha}}{N+1/2}\right)^\delta,
\end{equation}
where $ c $ is a constant not depending on $ \theta $ or $ n $. Hence the second term in \eqref{eq:section_proof_of_CUE:second_error} is of order $ \ln(n) \left(\frac{n^{1-\alpha}}{N+1/2}\right)^\delta \to 0 $ as $ n \to \infty $. So the second term in \eqref{eq:section_proof_of_CUE:l1_error} tends to zero, as $ n \to \infty $.
 This proves the result.
\end{proof}

Now we are ready for the
\begin{proof}[Proof of Theorem \ref{th:transCUEfirsttwocases}]
We will first prove the statement for $ f $ given in the beginning of section \ref{subsection:section_proof_of_CUE:approximation_to_an_analytic_function} and then extend the result to all functions satisfying the conditions in the theorem.

From Lemma \ref{lem:section_preliminaries:characteristic_function_as_integral}, Lemma \ref{lem:section_proof_of_CUE:approximation_of_integral} and Lemma \ref{lem:section_proof_of_CUE:approximated_characteristic_function} we conclude that
\begin{multline}\label{eq:section_proof_of_CUE:limit_for_arbitrary_parameters}
 \log \det\left(I - \left(1-\e^{\lambda n^{-s} f_{\alpha,N}}\right)K_{n,\gamma_n}\right) = \\
 = n\widehat{\log\varphi}_{n,1}(0) + \sum_{k=1}^{\infty}k\widehat{\log\varphi}_{n,1}(k)\widehat{\log\varphi}_{n,1}(-k) + \Ordo\left(\e^{\frac{n}{2N}\log\rho}\right).
\end{multline}
To understand the second term, we use Lemma \ref{lem:section_proof_of_CUE:varphitof}.
We deal with the first term in \eqref{eq:section_proof_of_CUE:limit_for_arbitrary_parameters} in two separate cases.

CASE 1. Let $ \delta > \alpha $, then $ s = 0 $. By \eqref{eq:section_proof_of_CUE:first_expansion_of_varphi} we get
\begin{equation}\label{eq:section_proof_of_CUE:error_of_first_fourier_coefficient}
 n|\widehat{\log\varphi_{n,1}}(0) - \lambda \hat{f}_{\alpha,N}(0)| \leq cn^{1-\delta}\int_{-\pi}^\pi|f_{\alpha,N}(\e^{i\theta})|\d\theta
\end{equation}
for some constant $ c $. By \eqref{eq:section_proof_of_CUE:order_of_l1} it follows that \eqref{eq:section_proof_of_CUE:error_of_first_fourier_coefficient} is of order $ n^{\alpha-\delta} $.

Since $ \delta > \alpha $, the error term obtained for \eqref{eq:section_proof_of_CUE:general_sobolev_term} and \eqref{eq:section_proof_of_CUE:error_of_first_fourier_coefficient} tends to zero as $ n $ tends to infinity. Hence
\begin{align}	
 & \lim_{n\rightarrow \infty} \left(\log \det\left(I - \left(1-\e^{\lambda f_{\alpha,N}}\right)K_{n,\gamma_n}\right) - \lambda n^{\alpha}\F(f)(0)\right) \\
 = & \lim_{n\rightarrow\infty}\frac{\lambda^2}{n^{1-\alpha}}\sum_{k=1}^N\frac{k}{n^{1-\alpha}}\F(f)\left(\frac{k}{n^{1-\alpha}}\right)\F(f)\left(\frac{-k}{n^{1-\alpha}}\right)  \\
 = &\frac{\lambda^2}{2}\int_{\RR}|\xi||\F(f)(\xi)|^2\d\xi,
\end{align}
where we in the last equality used that $ f $ is real valued and that the sum converges as a Riemann sum by the choice of $ N $, \eqref{eq:section_proof_of_CUE:def_N}, and since the derivative of $ \F(f) $ is bounded. By \eqref{eq:section_preliminaries:mean_of_thinned_CUE} and Lemma \ref{lem:section_proof_of_CUE:correct_approximation_of_f} the proof of Theorem \ref{th:transCUEfirsttwocases} for $ \delta > \alpha $ follows for functions satisfying \eqref{eq:section_proof_of_CUE:epsilon_bound}.

CASE 2. Let $ \delta < \alpha $ then $ s = \frac{\alpha-\delta}{2} $. The second term in \eqref{eq:section_proof_of_CUE:limit_for_arbitrary_parameters} tends to zero. For the first term, expand $ \log \varphi_{n,1} $ to
\begin{multline}
 \log \varphi_{n,1}(z) = \lambda n^{-(\alpha-\delta)/2}(1-\kappa n^{-\delta})f_{\alpha,N}(z) \\
 + \frac{\lambda^2}{2}n^{-\alpha}\kappa(1-\kappa n^{-\delta})f_{\alpha,N}(z)^2 + \Ordo(n^{-\delta}n^{-3(\alpha-\delta)/2}f_{\alpha,N}^3).
\end{multline}
Consider the zeroth Fourier coefficient of both sides. As in CASE 1 we may use \eqref{eq:section_proof_of_CUE:order_of_l1} to obtain the error term, namely
\begin{equation}
 n^{-\delta}n^{-3(\alpha-\delta)/2}\int_{-\pi}^\pi|f_{\alpha,N}(x)|^3\d x = \Ordo\left(n^{-\delta}n^{-3(\alpha-\delta)/2}n^{1-\alpha}\right).
\end{equation}
So
\begin{multline}
 n\widehat{\log\varphi_{n,1}}(0) = \lambda n^{1-(\alpha-\delta)/2}(1-\kappa n^{-\delta})\hat{f}_{\alpha,N}(0) \\
 + \frac{\lambda^2}{2}n^{1-\alpha}\kappa(1-\kappa n^{-\delta})\widehat{f^2}_{\alpha,N}(0) + \Ordo\left(n^{-\frac{1}{2}(\alpha-\delta)}\right).
\end{multline}
Hence, since $ \alpha > \delta $,
\begin{multline}
  \lim_{n\rightarrow \infty} \left(\log \det\left(I - \left(1-\e^{\lambda n^{-(\alpha-\delta)/2}f_{\alpha,N}}\right)K_{n,\gamma_n}\right) \right.\\
 \left.-\lambda (n^{(\alpha+\delta)/2}-\kappa n^{(\alpha-\delta)/2})\F(f)(0)\right) \\
 = \lim_{n\rightarrow \infty}\frac{\lambda^2}{2}n^{1-\alpha}\kappa(1-\kappa n^{-\delta})\widehat{f^2}_{\alpha,N}(0) \\
 = \frac{\lambda^2}{2}\frac{\kappa}{2 \pi} \|f\|^2_{\L^2(\RR)}.
\end{multline}
By \eqref{eq:section_preliminaries:mean_of_thinned_CUE} and Lemma \ref{lem:section_proof_of_CUE:correct_approximation_of_f} the result for $ \delta < \alpha $ follows for functions satisfying \eqref{eq:section_proof_of_CUE:epsilon_bound}.

To end the proof we need a final continuity argument to allow function $f$ satisfying \eqref{eq:optimalcondition} (or $\eps=0$ in \eqref{eq:section_proof_of_CUE:epsilon_bound}). Let $ f $ be a function satisfying the condition in Theorem \ref{th:transCUEfirsttwocases} which, since we are interested in large $ n $, we may assume have support in $ [-\pi,\pi] $. Denote the limiting distribution as $ X(f) $. By a variance estimate, as in the proof of Lemma \ref{lem:section_proof_of_CUE:correct_approximation_of_f}, we obtain, for a $ C^\infty $ function $ g $ with compact support in $ [-\pi,\pi] $, the inequality
\begin{multline}
 \left|\E\left[\e^{\i t n^{-s}\left(X_n^{(\alpha,\gamma_n)}(f)-\E[X_n^{(\alpha,\gamma_n)}(f)]\right)}\right] - \e^{- \frac{t^2}{2}\V[X(f)]}\right| \\
 \leq \left|\E\left[\e^{\i t n^{-s}\left(X_n^{(\alpha,\gamma_n)}(f)-\E[X_n^{(\alpha,\gamma_n)}(f)]\right)}\right] - \E\left[\e^{\i t n^{-s}\left(X_n^{(\alpha,\gamma_n)}(f)-\E[X_n^{(\alpha,\gamma_n)}(g)]\right)}\right]\right| \\
  + \left|\E\left[\e^{\i t n^{-s}\left(X_n^{(\alpha,\gamma_n)}(f)-\E[X_n^{(\alpha,\gamma_n)}(g)]\right)}\right] - \e^{-\frac{t^2}{2}\V[X(g)]}\right| \\ 
 + \left|\e^{-\frac{t^2}{2}\V[X(g)]}- \e^{-\frac{t^2}{2}\V[X(f)]}\right| \\
 \leq tn^{-s}\V[X_n^{(\alpha,\gamma_n)}(f-g)] + \frac{t^2}{2}\V[X(f-g)] \\
 + \left|\E\left[\e^{\i t n^{-s}\left(X_n^{(\alpha,\gamma_n)}(g)-\E[X_n^{(\alpha,\gamma_n)}(g)]\right)}\right] - \e^{-\frac{t^2}{2}\V[X(g)]}\right|. 
\end{multline}
From \eqref{eq:section_preliminaries:variance_of_thinned_CUE} and \eqref{eq:section_proof_of_CUE:H_bound} we get that $ \V[X_n^{(\alpha,\gamma_n)}(f)] $ is bounded by \eqref{eq:optimalcondition}, and the same is true for $ \V[X(f)] $. So by
choosing $ g $ such that $
 \int_\RR|\F(f-g)(\xi)|^2(1+|\xi|)\d\xi
$
is small, the first two terms are small. And since $ g $ satisfies \eqref{eq:section_proof_of_CUE:epsilon_bound}, the last term is small for large $ n $.
\end{proof}

\begin{proof}[Proof of Theorem \ref{th:transCUEthirdcase}]

Up to the evaluation of the first term in \eqref{eq:section_proof_of_CUE:limit_for_arbitrary_parameters} the proof is the same as in the proof of Theorem \ref{th:transCUEfirsttwocases} 

CASE 3. Let $ \delta = \alpha $ then $ s = 0 $. Expand $ \log \varphi_{n,1} $ to
\begin{equation}
 \log \varphi_{n,1}(z) = \lambda f_{\alpha,N}(z) + \sum_{m=1}^{\infty}\frac{(-1)^{m+1}}{m}\left(\kappa n^{-\alpha}(\e^{-\lambda f_{\alpha,N}(z)}-1)\right)^m .
\end{equation}
Again as in CASE 1 we may use \eqref{eq:section_proof_of_CUE:order_of_l1} to obtain
\begin{align}
 n\widehat{\log\varphi}_{n,1}(0) & = \lambda n\hat{f}_{\alpha,N}(0) + \frac{\kappa}{2\pi}n^{1-\alpha}\int_{-\pi}^{\pi}\left(\e^{-\lambda f_{\alpha,N}(\e^{\i\theta})}-1\right)\d\theta + \Ordo(n^{-\alpha}) \\
 & = \lambda n\hat{f_{\alpha}}(0) + \frac{\kappa}{2\pi}n^{1-\alpha}\int_{-\pi}^{\pi}\left(\e^{-\lambda f_{\alpha}(\theta)}-1\right)\d\theta + o(1)
\end{align}
where we used \eqref{eq:section_proof_of_CUE:order_of_l1} in the second equality. Hence
\begin{multline}
 \lim_{n\rightarrow \infty} \left(\log \det\left(I - \left(1-\e^{\lambda f_{\alpha,N}}\right)K_{n,\gamma_n}\right) - \lambda n^{\alpha}\F(f)(0)\right) \\
 = \lim_{n\rightarrow \infty}\frac{\kappa}{2\pi}n^{1-\alpha}\int_{-\pi}^{\pi}\left(\e^{-\lambda f_{\alpha}(\theta)}-1\right)\d\theta \\
 + \lim_{n\rightarrow\infty}\frac{\lambda^2}{n^{1-\alpha}}\sum_{k=1}^N\frac{k}{n^{1-\alpha}}\F(f)\left(\frac{k}{n^{1-\alpha}}\right)\F(f)\left(\frac{-k}{n^{1-\alpha}}\right)  \\
 = \frac{\kappa}{2\pi}\int_{\RR}\left(e^{-\lambda f(\theta)}-1\right)\d\theta + \frac{\lambda^2}{2}\int_{\RR}|\xi||\F(f)(\xi)|^2\d\xi.
\end{multline}
Hence Theorem \ref{th:transCUEthirdcase} follows by \eqref{eq:section_preliminaries:mean_of_thinned_CUE}.
\end{proof}

\section{Proof of Theorem \ref{th:transsinefirsttwocases} and \ref{th:transsinethirdcase}} \label{proof of Sine}

The proof of the CLT for the thinned sine process will have the same structure as the proof of the CLT for the thinned CUE. But with some simplifications. For instance, we will not use a truncation as in  \eqref{eq:section_proof_of_CUE:def_of_approximated_function}, but instead we will start with an analytic function. This also simplifies the calculation in Section \ref{subsection:section_proof_of_sine:the_asymptotic_behavior} compared with Section \ref{subsection:section_proof_of_CUE:the_asymptotic_behavior}. The opening of the lens in the steepest decent analysis  of the Riemann-Hilbert problem will also be easier.

We will  show there exists a disc around the origin such that
\begin{equation}
 \E\left[\e^{\lambda L^{-s}\left(X_L^{(\gamma_L)}(f)-\E[X_L^{(\gamma_L)}(f)]\right)}\right] \rightarrow M(\lambda)
\end{equation}
as $ L \rightarrow \infty $, where $ M $ is the moment generating function for the limiting distribution and  $ L^{-s} $, with $s= \max(0,\frac{1-\delta}{2})$,  is the normalizing factor needed to get a CLT. We will therefore consider
\begin{equation} \label{eq:moment-gen-fun-sine}
 \E\left[\e^{\lambda L^{-s}X_L^{(\gamma_L)}(f)}\right] = \det\left(I+(\e^{\lambda L^{-s}f_L}-1)K_{\text{sine},\gamma_L}\right)_{\L^2(\RR)}
\end{equation}
where $ f_L(x) = f\left(\frac{x}{L}\right) $.

\subsection{Approximation}\label{subsection:section_proof_of_sine:approximation}
We will prove Theorem \ref{th:transsinefirsttwocases} and \ref{th:transsinethirdcase} for functions of the form
\begin{equation}\label{eq:section_proof_of_sine:family_of_functions}
 f(x) = p(x)e^{-x^2}
\end{equation}
where $ p(x) $ is any polynomial. This is sufficient since they are dense in the space of functions with norm
\begin{equation}\label{eq:section_proof_of_sine:optimal_bound}
 \int_{\RR}|\F(f)(\xi)|^2(1+|\xi|)^{1+\epsilon}\d\xi < \infty
\end{equation}
for any $ \epsilon \geq 0 $. Moreover, by \eqref{eq:section_preliminaries:variance_of_Sine} and \eqref{eq:section_proof_of_CUE:sobolev_norm_in_two_ways}
\begin{equation}
 \V[X_L^{(1)}(f)] \leq 2\int_{\RR}|\F(f)(\xi)|^2|\xi|\d\xi.
\end{equation}
By Lemma \ref{lem:momenttovariance} this shows that the family $\left\{\E \left [\exp\left(\i t X_L^{(\gamma_L)}(f)-\E [X_L^{(\gamma_L)}(f)]\right)\right]\right\}_L$ is  equicontinuous in $f$ with respect to \eqref{eq:section_proof_of_sine:optimal_bound}. Also the limit as $L \to \infty$, in case $\delta \neq 1$, is continuous with respect to this norm.    Hence, an argument similar to the one in the end of the proof of Theorem \ref{th:transCUEfirsttwocases}, can be used to extend  Theorem \ref{th:transsinefirsttwocases} from  functions $ g $ of the form  
\eqref{eq:section_proof_of_sine:family_of_functions} to  functions $ f $ satisfying \eqref{eq:section_proof_of_sine:optimal_bound} with $\eps=0$. Indeed, if we denote the limiting distribution of $ X_L^{(\gamma_L)}(f) $ with $ X(f) $, we have 
\begin{multline}
 \left|\E\left[\e^{\i t L^{-s}\left(X_L^{(\gamma_L)}(f)-\E[X_L^{(\gamma_L)}(f)]\right)}\right] - \e^{- \frac{t^2}{2}\V[X(f)]}\right| \\
 \leq tL^{-s}\V[X_L^{(\gamma_L)}(f-g)] + \frac{t^2}{2}\V[X(f-g)] \\
 + \left|\E\left[\e^{\i t L^{-s}\left(X_L^{(\gamma_L)}(g)-\E[X_L^{(\gamma_L)}(g)]\right)}\right] - \e^{-\frac{t^2}{2}\V[X(g)]}\right| \\
 \leq \left(t+\frac{t^2}{2}\right)\int_{\RR}|\F(f-g)(\xi)|^2(1+|\xi|)\d\xi \\
 + \left|\E\left[\e^{\i t L^{-s}\left(X_L^{(\gamma_L)}(g)-\E[X_L^{(\gamma_L)}(g)]\right)}\right] - \e^{-\frac{t^2}{2}\V[X(g)]}\right|.
\end{multline}

Similarly, a variance estimate implies that it is enough to prove Theorem \ref{th:transsinethirdcase} for functions of the form \eqref{eq:section_proof_of_sine:family_of_functions}. However,  we need the strict inequality $\eps>0$ since the limit law for $\delta=1$ is not continuous with respect to \eqref{eq:section_proof_of_sine:optimal_bound} with $\eps=0$. 

\subsection{Riemann-Hilbert problem}\label{subsection:section_proof_of_sine:riemann-hilbert_problem}
In this section we will state and solve the Riemann-Hilbert problem related to the sine process. We will find the asymptotic leading term by a steepest descent method which later will be used to prove Lemma \ref{lem:section_proof_of_sine:approximation_of_integral}.

\begin{RHP}\label{RHP:section_proof_of_sine:first_problem}
Let $ \Gamma_m $ be the real line oriented from left to right. Find a function $ m:\Gamma_m \rightarrow \CC^{2\times 2} $ such that
\begin{itemize}
 \item $ m $ is analytic in $ \CC \backslash \Gamma_m $
 \item $ m_+(z) = m_-(z)
 \begin{pmatrix}
\varphi_{L,t}(z) & -(\varphi_{L,t}(z)-1)\e^{2\i z} \\
(\varphi_{L,t}(z)-1)\e^{-2\i  z} & 2 - \varphi_{L,t}(z)
\end{pmatrix}, $ $ z \in \Gamma_m $ 
 \item $ m(z) = I + \Ordo(z^{-1}) $ as $ |z| \rightarrow \infty $
\end{itemize}
where
\begin{equation}
\varphi_{L,t}(z) = 1 - t\gamma_L\left(1-\e^{\lambda L^{-s}f_L(z)}\right).
\end{equation}
\end{RHP}

Recall from Lemma \ref{lem:section_preliminaries:characteristic_function_as_integral} that we consider $ t \in [0,1] $.

The first step is a dilation, it is done to control how the Riemann-Hilbert problem varies with $ L $. 

Let $ \varphi_t(x) = \varphi_{L,t}(Lx) $, note that $ \varphi_t $ still depends on $ L $, but only from the thinning.
\begin{defn}[$m\rightarrow T$]\label{def:section_proof_of_sine:t_function}
  Let $ \Gamma_T = \Gamma_m $ and define $ T:\CC \backslash \Gamma_T \rightarrow \CC^{2\times 2} $ as
\begin{equation}
 T(z) = m(Lz).
\end{equation}
\end{defn}

\begin{lemma}
Let $ T $ be defined as in Definition \ref{def:section_proof_of_sine:t_function}, then $ T $ solves the Riemann-Hilbert problem
\begin{itemize}
 \item $ T $ is analytic in $ \CC \backslash \Gamma_T $
 \item $ T_+(z) = T_-(z)J_T(z) $ where 
 \begin{equation}
 J_T(z) = 
 \begin{pmatrix}
\varphi_t(z) & -(\varphi_t(z)-1)\e^{2\i L z} \\
(\varphi_t(z)-1)\e^{-2\i L z} & 2 - \varphi_t(z)
\end{pmatrix}, \quad z \in \Gamma_T
\end{equation}
\item $ T(z) = I + \Ordo(z^{-1}) $ as $ |z| \rightarrow \infty $
\end{itemize}
\end{lemma}
\begin{proof}
This follows directly since $ J_T(z) = J_m(Lz) $.
\end{proof}

\begin{figure}[t]
\begin{center}
\begin{tikzpicture}[scale=4]
\begin{scope}[decoration={markings,mark= at position 0.5 with {\arrow{stealth}}}]
\draw[very thick,postaction=decorate] (-1,0) -- (1,0);
\draw (0,0) node[above] {$+$};
\draw (0,0) node[below] {$-$};
\draw[very thick,postaction=decorate] (-1,0.3) -- (1,0.3);
\draw (0,0.3) node[above] {$+$};
\draw (0,0.3) node[below] {$-$};
\draw[very thick,postaction=decorate] (-1,-0.3) -- (1,-0.3);
\draw (0,-0.3) node[above] {$+$};
\draw (0,-0.3) node[below] {$-$};

\draw (-1,0.3) node[left]{$ \rho $};
\draw (1,0.3) node[right] {$
\begin{psmallmatrix}
1 & -(1 - \varphi_t(z)^{-1})\e^{2\i L z} \\
0 & 1
\end{psmallmatrix}
$};
\draw (-1,0) node[left]{$ 0 $};
\draw (1,0) node[right] {$
\begin{psmallmatrix}
\varphi_t(z) & 0 \\
0 & \varphi_t(z)^{-1}
\end{psmallmatrix}
$};
\draw(-1,-0.3) node[left] {$ -\rho $};
\draw (1,-0.3) node[right] {$
\begin{psmallmatrix}
1 & 0 \\
(1 - \varphi_t(z)^{-1})\e^{-2\i L z} & 1
\end{psmallmatrix}
$};
\end{scope}
\end{tikzpicture}
\end{center}
\caption{The contour and jump matrix for the Riemann-Hilbert problem stated in Lemma \ref{lem:section_proof_of_sine:RHP_s}.}
\end{figure}
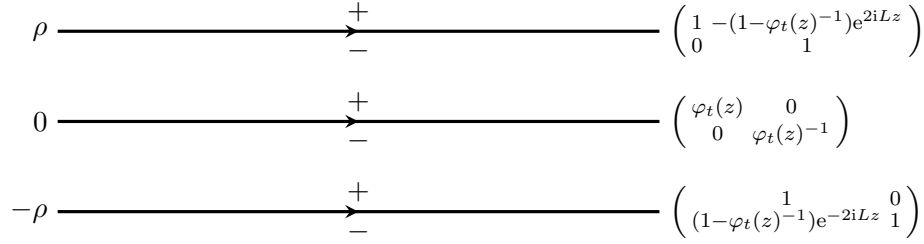
The next step is the opening of the (infinite) lens.
\begin{defn}[$ T \rightarrow S $]\label{def:section_proof_of_sine:s_function}
Let $ \rho > 0 $ and let $ \Gamma_S = \{z \in \CC; \im(z) \in \{-\rho,0,\rho \} \} $ oriented from left to right. Define $ S: \CC \backslash \Gamma_S \rightarrow \CC^{2\times2} $ as
\begin{align}
S(z) =& T(z), &\rho <& \im(z) \\
S(z) = & T(z)
\begin{pmatrix}
1 & -(1 - \varphi_t(z)^{-1})\e^{2 \i L z} \\
0 & 1
\end{pmatrix}^{-1},
 & 0 <& \im(z) < \rho \\
S(z) = & T(z)
\begin{pmatrix}
1 & 0 \\
(1 - \varphi_t(z)^{-1})\e^{-2 \i L z} & 1
\end{pmatrix},
& -\rho <& \im(z) < 0 \\
S(z) =& T(z), & & \im(z) < -\rho. \\
\end{align}
\end{defn}
\begin{lemma}\label{lem:section_proof_of_sine:RHP_s}
Let $ S $ be defined as in Definition \ref{def:section_proof_of_sine:s_function}, then $ S $ solves the Riemann-Hilbert problem
\begin{itemize}
 \item $ S $ is analytic in $ \CC \backslash \Gamma_S $
 \item $ S_+(z) = S_-(z)J_S(z) $ where 
\begin{align}
J_S(z) & =
\begin{pmatrix}
1 & -(1 - \varphi_t(z)^{-1})\e^{2 \i L z} \\
0 & 1
\end{pmatrix},
 & \im(z) = \rho \\
J_S(z)  & =
\begin{pmatrix}
\varphi_t(z) & 0 \\
0 & \varphi_t(z)^{-1}
\end{pmatrix}, 
& \im(z) = 0 \\
J_S(z) & =
\begin{pmatrix}
1 & 0 \\
(1 - \varphi_t(z)^{-1})\e^{-2 \i L z} & 1
\end{pmatrix},
& \im(z) = -\rho
\end{align}
\item $ S(z) = I + \Ordo(z^{-1}) $ as $ |z| \rightarrow \infty $
\end{itemize}
\end{lemma}
\begin{proof}
To make sure that $ S $ is well defined, we need that $ \varphi_t $ is nonzero in the strip $ -\rho \leq \im(z) \leq \rho $. This is holds as  long as $ |\lambda| $ is chosen small enough.

The jump condition follows by the factorization
\begin{equation}
J_T(z) = 
\begin{pmatrix}
1 & 0 \\
(1 - \varphi_t(z)^{-1})\e^{-2 \i L z} & 1
\end{pmatrix}
\begin{pmatrix}
\varphi_t(z) & 0 \\
0 & \varphi_t(z)^{-1}
\end{pmatrix}
\begin{pmatrix}
1 & -(1 - \varphi_t(z)^{-1})\e^{2 \i L z} \\
0 & 1
\end{pmatrix}.
\end{equation}
By the choice of $ f $, \eqref{eq:section_proof_of_sine:family_of_functions}, it is clear that $ 1-\varphi_t(z)^{-1} \to  0 $ in the strip  $ -\rho < \im(z) < \rho $. Hence $ T \to I $  as $ z\to \infty$ . This finishes the proof. 
\end{proof}

The idea of the transformation $ T \rightarrow S $ is that the jumps on $ \im(z) = \rho $ and $ \im(z) =-\rho $ tends to the identity matrix as $ L $ tends to infinity and the jump on the real line defines a Riemann-Hilbert problem that we will solve for finite $ L $.
\begin{defn}[Global Parametrix]\label{def:section_proof_of_sine:p_function}
Let $ \Gamma_P = \RR $ oriented from left to right. Define $ P:\CC\backslash \Gamma_p \rightarrow \CC^{2\times 2} $ as
\begin{equation}
P(z) =
\begin{pmatrix}
\e^{C\log \varphi_t(z)} & 0 \\
0 & \e^{-C\log\varphi_t(z)}
\end{pmatrix}.
\end{equation}
where $ C $ is the Cauchy operator on the real line, defined for $ h \in \L^2(\RR) $ and given by
\begin{equation}\label{eq:section_proof_of_sine:cauchy_operator_real_line_plus}
 Ch(z) = \int_0^{\infty}\F(h)(\xi)e^{i\xi z}\d\xi
\end{equation}
for $ \im(z) > 0 $ and
\begin{equation}\label{eq:section_proof_of_sine:cauchy_operator_real_line_minus}
 Ch(z) = -\int_0^{\infty}\F(h)(-\xi)e^{-i\xi z}\d\xi
\end{equation}
for $ \im(z) < 0 $.
\end{defn}
Note that $ \log \varphi_t $ is well defined if we make sure that $ \lambda $ is small enough.
\begin{lemma}
Let $ P $ be defined as in Definition \ref{def:section_proof_of_sine:p_function}, then $ P $ solves the Riemann-Hilbert problem 
\begin{itemize}
 \item $ P $ is analytic in $ \CC \backslash \Gamma_P $
 \item $ P_+(z) = P_-(z)
 \begin{pmatrix}
\varphi_t(z) & 0 \\
0 & \varphi_t(z)^{-1}
\end{pmatrix}, \quad z \in \Gamma_T $
\item $ T(z) = I + \Ordo(z^{-1}) $ as $ |z| \rightarrow \infty $.
\end{itemize}
\end{lemma}
\begin{proof}
 Follows directly from the Fourier inversion theorem.
\end{proof}

We expect that $S$ and $P$ are close. In the last transformation we study  the error term. 
\begin{defn}[$ S \rightarrow R $]\label{def:section_proof_of_sine:r_function}
 Let $ \Gamma_R = \Gamma_S $ and define $ R:\CC \backslash \Gamma_R \rightarrow \CC^{2\times 2} $ as
 \begin{equation}
  R(z) = S(z)P(z)^{-1}.
 \end{equation}
\end{defn}

\begin{lemma}\label{lem:section_proof_of_sine:solution_R_function}
Define $ R $ as in Definition \ref{def:section_proof_of_sine:r_function}, then $ R $ solves the Riemann-Hilbert problem
\begin{itemize}
 \item $ R $ is analytic in $ \CC \backslash \Gamma_R $
 \item $ R_+(z) = R_-(z)J_R(z) $ where 
 \begin{align}
  J_R(z) & = I, & \im(z) & = 0 \\
  J_R(z) & = P(z)J_S(z)P(z)^{-1}, & \im(z) & = \pm\rho
\end{align}
\item $ R(z) = I + \Ordo(z^{-1}) $ as $ |z| \rightarrow \infty$.
\end{itemize}
Moreover,
\begin{equation}
 |R(z) -I|\leq \frac{c}{\rho-|\im(z)|}\e^{-2 L \rho}
\end{equation}
for some constant $ c $.
\end{lemma}
\begin{proof}
The structure of the proof is the same as in the proof of Lemma \ref{lem:section_proof_of_CUE:solution_R_function}. The differences are that we are working on the real line instead of the unit circle and consider a different scaling. In fact, these differences, especially the scaling, simplifies the analysis. We will therefore be brief with the details.

Following the same strategy as in the proof of Lemma \ref{lem:section_proof_of_CUE:solution_R_function}, leads to the estimate
\begin{equation}
\|J_R - I \|_{\L^2(\Gamma_R) \cap \L^\infty(\Gamma_R)} \leq c' \e^{-2 \rho L}\\
\end{equation}
for some constant $ c' $. As in the proof of Lemma \ref{lem:section_proof_of_CUE:solution_R_function} we can relay on general principles and find a solution using the Cauchy operator and a Neumann series. It implies that for any  $ z \in \CC $ with $ |\im(z)| < \rho $
\begin{align}
|(R-I)(z)| & \leq \int_{\Gamma_R \backslash \RR} \frac{|\mu(w)||(J_R - I)(w)|}{|w-z|} |\d w| \\
 & \leq \left(\left(\int_{\Gamma_R\backslash \RR}\frac{|\d w|}{|w-z|^2}\right)^{1/2} + \frac{\|\mu-I\|_{\L^2(\Gamma_R)}}{\rho - |\im(z)|}\right) \|J_R - I \|_{\L^2(\Gamma_R)} \\
& \leq \frac{c}{\rho - |\im(z)|} \e^{-2 \rho L},
\end{align}
for some constant $ c $. Here $ \mu - I \in \L^2(\Gamma_R) $ comes from the general considerations  for Riemann-Hilbert problems, see \cite{Deift00}, Theorem 7.103.
\end{proof}

\subsection{Asymptotic behavior of the moment generating function}\label{subsection:section_proof_of_sine:the_asymptotic_behavior}
In this section we will find and calculate the leading term of the moment generating function. The leading term is given in Lemma \ref{lem:section_proof_of_sine:approximation_of_integral} and is attained from the the global parametrix given in last section.

In the same way as for the CUE, Lemma \ref{lem:section_proof_of_sine:solution_R_function} gives us the unique solution to the Riemann-Hilbert problem by tracing back $ R \rightarrow S \rightarrow T \rightarrow m $. By \eqref{eq:section_preliminaries:capital_f_one} and \eqref{eq:section_preliminaries:capital_f_two},
\begin{equation}\label{eq:section_proof_of_sine:relation_between_f_and_g}
F^{(1)}(z) = G^{(1)}(z) + \left(R_+\left(\frac{z}{L}\right)-I\right)G^{(1)}(z)
\end{equation}
and 
\begin{equation}
F^{(2)}(z) = G^{(2)}(z) + \left(\left(R_+\left(\frac{z}{L}\right)^{-1}\right)^T-I\right)G^{(2)}(z)
\end{equation}
where $ G^{(1)} $ and $ G^{(2)} $ are the leading term of $ F^{(1)} $ and $ F^{(2)} $ attained by the global parametrix and are given by
\begin{equation}\label{eq:section_proof_of_sine:capital_g_one}
G^{(1)}(z) = \frac{1}{2\pi \i}\left(\e^{(C_-\log \varphi_{L,t})(z)} \e^{\i z}, \e^{-(C_+\log \varphi_{L,t})(z)}\e^{-\i z} \right)^T
\end{equation}
and
\begin{equation}\label{eq:section_proof_of_sine:capital_g_two}
G^{(2)}(z) = \left(\e^{-(C_+\log \varphi_{L,t})(z)}\e^{-\i z}, -\e^{(C_-\log \varphi_{L,t})(z)}\e^{\i z}\right)
\end{equation}
where $ C_+ $ and $ C_- $ are defined as \eqref{eq:section_proof_of_sine:cauchy_operator_real_line_plus} and \eqref{eq:section_proof_of_sine:cauchy_operator_real_line_minus} respectively but with $ z \in \RR $.
\begin{lemma}\label{lem:section_proof_of_sine:approximation_of_integral}
 Let $ F^{(1)} $, $ F^{(2)} $, $ G^{(1)} $ and $ G^{(2)} $ be defined as in \eqref{eq:section_preliminaries:capital_f_one}, \eqref{eq:section_preliminaries:capital_f_two}, \eqref{eq:section_proof_of_sine:capital_g_one} and \eqref{eq:section_proof_of_sine:capital_g_two} respectively. Then
\begin{multline}
 \int_0^1\frac{1}{t} \int_{\RR}(\varphi_{L,t}(x) - 1)F^{(1)'}(x)^T F^{(2)}(x)\d x \d t = \\
 = \int_0^1\frac{1}{t}\int_{\RR}(\varphi_{L,t}(x) - 1)G^{(1)'}(x)^T G^{(2)}(x) \d x \d t + \Ordo\left(L\e^{-2\rho L}\right)
\end{multline}
 as $ L $ tends to infinity.
\end{lemma}
\begin{proof}
By extending $ G^{(1)} $ to an analytic function over the real line, it is clear, from the maximum modulus principle that $ G^{(1)} $ is uniformly bounded for $ -\rho \leq \im(z) \leq \rho $. We extend $ F^{(1)} - G^{(1)} $ to an analytic function for $ -\rho < \im(z) < \rho $. By doing that and by \eqref{eq:section_proof_of_sine:relation_between_f_and_g} and Lemma \ref{lem:section_proof_of_sine:solution_R_function} we have, for $ |\im(z)| \leq \frac{\rho}{2} $,
\begin{align}
 |F^{(1)}(z) - G^{(1)}(z)| & \leq \frac{c}{\rho-\frac{\rho}{2L}}\e^{-2 L \rho}|G^{(1)}(z)| \\
 & \leq c'\e^{-2 L \rho}
\end{align}
for some constant $ c' $. By Cauchy integral formula, 
\begin{equation}
|F^{(1)'}(z) - G^{(1)'}(z)| \leq \pi \rho c'e^{-2 L \rho}
\end{equation}
for $ z \in \RR $. In a similar way a similar estimation for the difference between $ F^{(2)} $ and $ G^{(2)} $ is obtained. Hence
\begin{align}
& \left|\int_{\RR}(\varphi_{L,t}(x)-1) \left(F^{(1)'}(x)F^{(2)}(x) - G^{(1)'}(x)G^{(2)}(x)\right)\d x\right| \\
\leq & \|F^{(1)'}F^{(2)} - G^{(1)'}G^{(2)}\|_{\L^{\infty}(\RR)}L\int_{\RR}|(\varphi_t(x)-1)|\d x \\
\leq & c'' tL \e^{-2 L\rho}.
\end{align}
for some constant $ c'' $. 
\end{proof}

\begin{lemma}\label{lem:section_proof_of_sine:approximated_characteristic_function}
Let $ G^{(1)} $ and $ G^{(2)} $ be defined as in Lemma \ref{lem:section_proof_of_sine:approximation_of_integral} and recall that we use the notation $ \log\varphi_1(x) = \log\varphi_{L,1}(Lx) $, then
\begin{align}
& -\int_0^1 \frac{1}{t}\int_{\RR}(\log\varphi_{L,t}(x)-1)G^{(1)\prime}(x)^T G^{(2)}(x)\d x \d t \\
& = \frac{L}{\pi}\int_{\RR}\log\varphi_1(\xi) \d\xi + \int_0^{\infty} \xi \F (\log\varphi_1)(\xi)\F(\log\varphi_1)(-\xi) \d\xi.
\end{align}
\end{lemma}
\begin{proof}
A straightforward differentiation of $ G^{(1)}(x) $ and simplification yields that
\begin{align}
& -\int_0^1 \frac{1}{t}\int_{\RR}(\varphi_{L,t}(x)-1)G^{(1)\prime}(x)^T G^{(2)}(x)\d x \d t \\
& = \frac{1}{2\pi \i}\int_0^1 \int_{\RR}\frac{\varphi_{L,1}(x)-1}{\varphi_{L,t}(x)}\left((C_+\log \varphi_{L,t})'(x) + (C_-\log\varphi_{L,t})'(x) + 2 \i\right)\d x \d t.
\end{align}
For the last term, we use Fubini's theorem and a change of variable, to obtain
\begin{equation}
 \frac{1}{2\pi \i}\int_0^1 \int_{\RR}\frac{\varphi_{L,1}(x)-1}{\varphi_{L,t}(x)}2\i\d x \d t = \frac{L}{\pi}\int_{\RR}\log\varphi_1(\xi) \d\xi.
\end{equation}
For the first and second term, we use the fact that the Cauchy operator on the real line can be expressed by the Fourier transform, that is,
\begin{equation}
 C_+ = \F^{-1}\circ1_{(\infty,0]}\F
\end{equation}
and
\begin{equation}
 C_- = -\F^{-1}\circ1_{[0,\infty)}\F.
\end{equation}
By Fubini's theorem
\begin{align}
 &  \frac{1}{2\pi \i}\int_0^1\int_{\RR}\frac{\varphi_{L,1}(x)-1}{\varphi_{L,t}(x)}\left((C_+\log\varphi_{L,t})'(x) + (C_-\log\varphi_{L,t})'(x)\right)\d x\d t\\
 = & \frac{1}{2\pi \i}\int_0^1\int_{\RR}\frac{\partial }{\partial t }\log\varphi_{L,t}(x)\int_0^{\infty} \xi \left(\F(\log\varphi_{L,t})(\xi)\e^{\i\xi x}+\F(\log\varphi_{L,t})(-\xi)\e^{-\i\xi x}\right)\d\xi \d x\d t \\
 = & \int_0^1\int_0^{\infty} \xi \left(\F(\log\varphi_{L,t})(\xi) \frac{\partial }{\partial t }\frac{1}{2\pi \i}\int_{\RR} \log\varphi_{L,t}(x)\e^{\i\xi x}\d x\right. \\
 & + \left.\F(\log\varphi_{L,t})(-\xi)\frac{\partial }{\partial t }\frac{1}{2\pi \i}\int_{\RR} \log\varphi_{L,t}(x)\e^{-\i\xi x}\right)\d x \d\xi \d t \\
 = & \int_0^{\infty} \xi \F(\log\varphi_{L,1})(\xi)\F(\log\varphi_{L,1})(-\xi) \d\xi.
\end{align}
Note that
\begin{equation}
\F(\log\varphi_{L,1})(\xi) = L \F(\log\varphi_1)(L\xi),
\end{equation}
so by a change of variable we get the result of the lemma.
\end{proof}

\begin{proof}[Proof of Theorem \ref{th:transsinefirsttwocases}]
 By Lemma \ref{lem:section_preliminaries:characteristic_function_as_integral}, Lemma \ref{lem:section_proof_of_sine:approximation_of_integral} and Lemma \ref{lem:section_proof_of_sine:approximated_characteristic_function} we know that there exists a disc around the origin such that for $ \lambda $ inside that disc
 \begin{multline}
  \log \det\left(I+\left(\e^{\lambda L^{-s} f_L}-1\right)K_{\text{sine},\gamma_L}\right) = \\ = \frac{L}{\pi}\int_{\RR}\log\varphi_1(\xi) \d\xi + \int_0^{\infty} \xi \F (\log\varphi_1)(\xi)\F(\log\varphi_1)(-\xi) \d\xi + \Ordo(L\e^{-2\rho L})\label{eq:section_proof_of_sine:limit_for_arbitrary_parameters}
 \end{multline}
 as $ L \rightarrow \infty $. Hence by \eqref{eq:moment-gen-fun-sine}, to understand the asymptotic behavior of the moment generating function, we need to understand \eqref{eq:section_proof_of_sine:limit_for_arbitrary_parameters} asymptotically as $L\to \infty$. 
  
 For the second term in \eqref{eq:section_proof_of_sine:limit_for_arbitrary_parameters} we will use the following inequality, which follows from Cauchy-Schwarz inequality, 
 \begin{align}
  & \left|\int_0^{\infty}\xi\left(\F(\log\varphi_1)(\xi)\F(\log\varphi_1)(-\xi) - \F(\lambda L^{-s} f)(\xi)\F(\lambda L^{-s} f)(-\xi)\right)\d\xi\right| \\
  \leq & \left(\left(\int_0^{\infty}|\xi \F(\log\varphi_1)(\xi)|^2\d\xi\right)^{\frac{1}{2}} + \left(\int_0^{\infty}|\xi \F(\lambda L^{-s} f)(-\xi)|^2\d\xi\right)^{\frac{1}{2}}\right) \\
  & \times\|\log\varphi_1 - \lambda L^{-s} f\|_{\L^2(\RR)}.
 \end{align}
 The right hand side tends to zero as $ L $ tends to infinity by Lebesgue dominated convergence theorem.
 
 For the first term in \eqref{eq:section_proof_of_sine:limit_for_arbitrary_parameters} we consider distinct cases.
 
 CASE 1. Let $ \delta > 1 $ then $ s = 0 $. Write
 \begin{equation}
  \log\varphi_1(\xi) = \lambda f(\xi) + \log\left(1+\kappa L^{-\delta}\left(\e^{-\lambda f(\xi)} - 1\right)\right).
 \end{equation}
 Then
 \begin{equation}
  \frac{L}{\pi}\int_\RR \log \varphi_1(\xi)\d\xi = \lambda \frac{L}{\pi}\int_\RR f(\xi)\d\xi + \Ordo(L^{1-\delta}).
 \end{equation}
 We can now, after subtracting the mean \eqref{eq:section_preliminaries:mean_of_thinned_Sine}, let $ L $ tend to infinity in \eqref{eq:section_proof_of_sine:limit_for_arbitrary_parameters}. Hence
 \begin{equation}
  \log \det\left(I+\left(\e^{\lambda f_L}-1\right)K_{\text{sine},\gamma_L}\right)- \lambda\frac{L \gamma}{\pi}\int_\RR f(\xi)\d\xi \rightarrow \frac{\lambda^2}{2}\int_{\RR} |\xi| |\F (f)(\xi)|^2 \d\xi
 \end{equation}
 as $ L \rightarrow \infty $ which proves the result.

 CASE 2. Let $ \delta < 1 $ then $ s = \frac{1-\delta}{2} $. That $ s > 0 $ implies that it is only the first term in \eqref{eq:section_proof_of_sine:limit_for_arbitrary_parameters} that gives a contribution in the limit. By the expansion
 \begin{equation}
  \log\varphi_1(\xi) = \lambda \gamma_L L^{-\frac{1-\delta}{2}}f(\xi) + \frac{1}{2}\lambda^2 \gamma_L \kappa L^{-1} f(\xi)^2 + \Ordo \left(L^{-1 + \frac{1-\delta}{2}}\right)
 \end{equation}
 we see that
 \begin{equation}
  \frac{L}{\pi}\int_\RR \log \varphi_1(\xi)\d\xi = \lambda \gamma_L L^{\frac{1+\delta}{2}}\int_\RR f(\xi)\d\xi + \frac{1}{2}\lambda^2 \gamma_L \kappa \int_\RR f(\xi)^2\d\xi + \Ordo\left(L^{\frac{1-\delta}{2}}\right).
 \end{equation}
 Hence, by subtracting the mean \eqref{eq:section_preliminaries:mean_of_thinned_Sine} in \eqref{eq:section_proof_of_sine:limit_for_arbitrary_parameters}, we obtain
 \begin{equation}
  \log \det\left(I+\left(\e^{\lambda L^{-s} f_L}-1\right)K_{\text{sine},\gamma_L}\right) - \lambda \gamma_L L^{\frac{1+\delta}{2}}\int_\RR f(\xi)\d\xi \rightarrow \frac{1}{2}\lambda^2\frac{\kappa}{\pi} \int_\RR f(\xi)^2\d\xi
 \end{equation}
 as $ L \rightarrow \infty $ which proves the result.
 \end{proof}
\begin{proof}[Proof of Theorem \ref{th:transsinethirdcase}]
Up to the evaluation of the first term in \eqref{eq:section_proof_of_sine:limit_for_arbitrary_parameters} the proof is the same as in the proof of Theorem \ref{th:transsinefirsttwocases}.

CASE 3. Let $ \delta = 1 $ then $ s = 0 $. Expand $ \log \varphi_1 $ to
\begin{equation}
 \log \varphi_1(\xi) = \lambda f(\xi) + \sum_{m=1}^{\infty}\frac{(-1)^{m+1}}{m}\left(\kappa L^{-1}(\e^{-\lambda f(\xi)}-1)\right)^m.
\end{equation}
Then 
\begin{equation}
 \frac{L}{\pi}\int_\RR \log \varphi_1(\xi)\d\xi = \lambda \frac{L}{\pi}\int_{\RR}f(\xi)\d\xi + \frac{\kappa}{\pi}\int_{\RR}\left(\e^{-\lambda f(x)}-1\right)\d\xi + \Ordo(L^{-1}).
\end{equation}
By subtracting the mean \eqref{eq:section_preliminaries:mean_of_thinned_Sine} in \eqref{eq:section_proof_of_sine:limit_for_arbitrary_parameters} we can conclude that
\begin{multline}
 \log \det\left(I+\left(\e^{\lambda f_L}-1\right)K_{\text{sine},\gamma_L}\right) - \lambda \frac{ \gamma_L L}{\pi}\int_{\RR}f(\xi)\d\xi \\
 \rightarrow \frac{\lambda^2}{2}\int_{\RR} |\xi| |\F (f)(\xi)|^2 \d\xi + \frac{\kappa}{\pi}\int_{\RR}\left(\e^{-\lambda f(x)}-1+ \lambda f(x) \right)\d\xi
\end{multline}
as $ L \rightarrow \infty $, which proves the result.
\end{proof}

\end{document}